\documentclass{amsart}

\usepackage{color}
\usepackage{mathrsfs}
\usepackage{amssymb}
\usepackage{amsmath}
\usepackage{amsfonts}
\usepackage{stmaryrd}
\usepackage{paralist}

\newtheorem*{thma}{Theorem~A}
\newtheorem*{thmb}{Theorem~B}
\newtheorem*{thmc}{Theorem~C}
\newtheorem*{thmd}{Theorem~D}
\newtheorem{theorem}{Theorem}[section]
\newtheorem{lemma}[theorem]{Lemma}
\newtheorem{proposition}[theorem]{Proposition}
\newtheorem{corollary}[theorem]{Corollary}
\newtheorem{claim}[theorem]{Claim}

\newtheorem{fact}[theorem]{Fact}
\newtheorem{question}[theorem]{Question}

\theoremstyle{definition}
\newtheorem{definition}[theorem]{Definition}
\newtheorem{remark}[theorem]{Remark}

\usepackage[pdftex]{hyperref}
\hypersetup{
    colorlinks=true, 
    linktoc=all,     
    linkcolor=blue,  
    allcolors=blue,
}

\usepackage[framemethod=tikz]{mdframed}

\newcommand{\dom}[1]{\ensuremath{\mathrm{dom}}(#1)}

\newcommand{\power}{\ensuremath{\mathscr{P}}}
\newcommand{\set}[2]{\ensuremath{\{#1 \,|\, #2 \}}}

\newcommand{\mc}{\mathcal}

\newcommand{\bb}{\mathbb}

\newcommand{\beq}{\begin{equation}}
\newcommand{\eeq}{\end{equation}}
\newcommand{\brm}{\begin{remark}\begin{rm}}
\newcommand{\erm}{\end{rm}\end{remark}}

\newcommand{\bce}{\begin{compactenum}}
\newcommand{\ece}{\end{compactenum}}

\newcommand{\cf}{\mathrm{cf}}

\newcommand{\otp}{\mathrm{otp}}
\newcommand{\height}{\mathrm{ht}}

\newcommand{\pp}{\mathrm{pp}}

\newcommand{\R}{\bb{R}}
\newcommand{\Q}{\bb{Q}}
\renewcommand{\P}{\bb{P}}

\newcommand{\M}{\bb{M}}

\newcommand{\C}{\bb{C}}
\newcommand{\Coll}{\mathrm{Coll}}

\newcommand{\SP}{{\sf SP}}
\newcommand{\ISP}{{\sf ISP}}

\newcommand{\ZFC}{\sf ZFC}

\newcommand{\SCH}{\sf SCH}

\newcommand{\PFA}{\sf PFA}
\newcommand{\MA}{\sf MA}
\newcommand{\CP}{\sf CP}
\newcommand{\IGMP}{\sf IGMP}
\newcommand{\wKH}{\sf wKH}
\newcommand{\AGP}{\mathsf{AGP}}

\newcommand{\wAGP}{\mathsf{wAGP}}
\newcommand{\GMP}{\mathsf{GMP}}
\newcommand{\SSH}{\mathsf{SSH}}

\begin{document}

\title{Guessing models, trees, and cardinal arithmetic}

\author{Chris Lambie-Hanson}
\address[Lambie-Hanson]{
Institute of Mathematics, 
Czech Academy of Sciences, 
{\v Z}itn{\'a} 25, Prague 1, 
115 67, Czech Republic
}
\email{lambiehanson@math.cas.cz}
\urladdr{https://users.math.cas.cz/~lambiehanson/}

\author{{\v S}{\'a}rka Stejskalov{\'a}}
\address[Stejskalov{\'a}]{
Charles University, Department of Logic,
Celetn{\' a} 20, Prague~1, 
116 42, Czech Republic
}
\email{sarka.stejskalova@ff.cuni.cz}
\urladdr{logika.ff.cuni.cz/sarka}

\address{Institute of Mathematics, Czech Academy of Sciences, {\v Z}itn{\'a} 25, Prague 1, 115 67, Czech Republic}

\thanks{Both authors were supported by the Czech Academy of Sciences 
(RVO 67985840) and the GA\v{C}R project 23-04683S}

\begin{abstract}
  Since being isolated by Viale and Wei\ss\ in 2009, the Guessing Model Property has emerged as 
  a particularly prominent and powerful consequence of the Proper Forcing Axiom. In this paper, 
  we investigate connections between variations of the Guessing Model Property and cardinal 
  arithmetic, broadly construed. We improve upon results of Viale and Krueger by proving that 
  a weakening of the Guessing Model Property implies Shelah's Strong Hypothesis. We also prove 
  that, though the Guessing Model Property is known not to put an upper bound on the size of the 
  continuum, it does imply that $2^{\omega_1}$ is as small as possible relative to the value of 
  $2^\omega$. Building on work of Laver, we prove that, in the extension of any model of $\PFA$ 
  by a measure algebra, every tree of height and size $\omega_1$ is B-special (a generalization of 
  specialness introduced by Baumgartner that can also hold of trees with uncountable branches). 
  Finally, we investigate the impact of forcing axioms for Suslin and almost Suslin 
  trees on guessing model properties. In particular, we prove that
  if $S$ is a Suslin tree, then the axioms $\PFA(S)$ and 
  $\PFA(S)[S]$ imply the Guessing Model Property and the Indestructible Guessing Model Property, 
  respectively, and, if $T^*$ is an almost Suslin Aronszajn tree, then the axiom 
  $\PFA(T^*)$ implies the Indestructible Guessing Model Property. This answers a number of questions of Cox and Krueger.
\end{abstract}

	\maketitle

\tableofcontents

\section{Introduction}

A long line of research in set theory, dating back at least to the 1970s, concerns the study of the extent 
to which certain combinatorial properties of large cardinals can consistently hold at smaller cardinals and 
of the consequences and implications of such properties. Classical examples of such properties include 
stationary reflection principles and the tree property; particularly prominent in recent years have been 
two-cardinal tree properties, which, in their strongest forms, can usefully be reformulated as guessing model 
principles.

The study of two-cardinal tree properties began with work of Jech \cite{jech_combinatorial_problems} and 
Magidor \cite{magidor_combinatorial_characterization}, who used them to give combinatorial characterizations 
of strongly compact and supercompact cardinals, respectively. Versions of two-cardinal tree properties 
consistent at small cardinals were first isolated by Wei\ss\ \cite{weiss} in the late 2000s, and since then 
much work has been done to understand these principles. In \cite{viale_weiss}, Viale and Wei\ss\ show that 
the strongest of these two-cardinal tree properties, $\ISP$, can be reformulated in terms of the existence 
of \emph{guessing models}; this reformulation has proven to be quite fruitful in deriving consequences of 
two-cardinal tree properties. Viale and Wei\ss\ also prove that the guessing model property at $\omega_2$, 
which is equivalent to $\ISP_{\omega_2}$ and often simply denoted $\GMP$, follows from the 
Proper Forcing Axiom ($\PFA$).

In this article, we continue the study of two-cardinal tree properties and guessing model principles, 
focusing in particular on their relationship to cardinal arithmetic, broadly construed. There has already 
been considerable work done in this direction. Notably, results of Viale \cite{viale_guessing_models} and 
Krueger \cite{krueger_mitchell_iteration} show that $\GMP$ implies the Singular Cardinals Hypothesis 
($\SCH$), and, more generally, that for any regular cardinal $\kappa \geq \omega_2$, 
$\ISP_\kappa$ implies $\SCH$ above $\kappa$. 
On the other hand Cox and Krueger prove in \cite{cox_krueger_quotients} that, unlike $\PFA$, 
$\GMP$ does not place any restrictions on the value of the continuum other than implying that $2^\omega > 
\omega_1$.

In the first part of the paper, we continue these lines of investigation. In particular, we show that, 
for a regular cardinal $\kappa \geq \omega_2$, a certain weakening of $\ISP_\kappa$ implies 
Shelah's Strong Hypothesis ($\SSH$) above $\kappa$. This improves upon Viale and Krueger's aforementioned 
result by weakening the hypotheses and strengthening the conclusion. (The terminology and notation 
in the statement of Theorem A will be introduced in later sections; for now, we note that the 
hypothesis of the theorem is a weakening of $\ISP_\kappa$.)

\begin{thma}
  Let $\kappa \geq \omega_2$ be a regular cardinal and,
  for every singular cardinal $\lambda > \kappa$ of countable cofinality, let 
  \[
    \mc Y_\lambda := \{M \in \power_\kappa H(\lambda^{++}) \mid M \text{ is } 
    (\omega_1, \lambda)\text{-internally unbounded}\}.
  \]
  If $\wAGP_{\mc Y_\lambda}(\kappa, \kappa, \lambda^{++})$ holds for all such $\lambda$,
  then $\SSH$ holds above $\kappa$.
\end{thma}

In the process of proving Theorem A, we sharpen some results of Viale 
\cite{viale_covering} regarding covering properties. In particular, 
in \cite{viale_covering} and \cite{viale_guessing_models}, Viale shows that 
various consequences of $\PFA$, including $\GMP$, imply certain covering properties 
at all singular cardinals of countable cofinality above $2^\omega$. Here, 
we improve upon this by removing the assumption that the singular cardinal is 
greater than the continuum. See Theorem \ref{cp_thm} below for a more precise 
statement; the main new technical lemma leading to this improvement is 
Lemma \ref{downward_coherence_lemma}.

By the work of Cox and Krueger \cite{cox_krueger_quotients}, $\GMP$ is compatible with $\cf(2^\omega) = 
\omega_1$ and therefore, unlike $\PFA$, or even $\MA_{\omega_1}$, 
it does not imply that $2^{\omega_1} = 2^\omega$. 
However, as a consequence of the proof Theorem A, we are able to prove that 
a weakening of $\GMP$ implies that $2^{\omega_1}$ is as small as possible relative to the value of $2^\omega$. We 
additionally show that the negation of the weak Kurepa Hypothesis ($\neg \wKH$), which is a consequence of 
the same weakening of $\GMP$, yields the same conclusion if $2^\omega < \aleph_{\omega_1}$ but can 
consistently fail to do so if $2^\omega \geq \aleph_{\omega_1}$.

\begin{thmb}
  \begin{enumerate}
    \item If $\wAGP(\omega_2)$ holds and $\wAGP_{\mc Y_\lambda}(\lambda^{++})$ holds for all singular 
    $\lambda < 2^\omega$ of countable cofinality, where 
	\[
    \mc Y_\lambda := \{M \in \power_{\omega_2} H(\lambda^{++}) \mid M \text{ is } 
    (\omega_1, \lambda)\text{-internally unbounded}\},
  \]    
 then 
    \[
      2^{\omega_1} = \begin{cases}
        2^\omega & \text{if } \cf(2^\omega) \neq \omega_1 \\
        (2^\omega)^+ & \text{if } \cf(2^\omega) = \omega_1.
      \end{cases}
    \]
    \item If $\neg\wKH$ holds and $2^{\omega} < \aleph_{\omega_1}$, then $2^{\omega_1} = 2^\omega$.
    \item Relative to the consistency of the existence of a supercompact cardinal, it is consistent 
    that $\neg\wKH$ holds, $2^\omega = \aleph_{\omega_1}$, but $2^{\omega_1} > \aleph_{\omega_1+1}$.
  \end{enumerate}
\end{thmb}

In the second half of the paper, we prove some results motivated in part by the indestructible guessing 
model property ($\IGMP$) introduced by Cox and Krueger in \cite{cox_krueger_indestructible}, and which 
we also feel are of independent interest. We first prove a result about special trees. Recall that a 
tree of height $\omega_1$ is \emph{special} if it is the union of countably many antichains. Special 
trees therefore have no uncountable branches. In 
\cite{baumgartner_pfa}, Baumgartner introduces a generalization of specialness, which we call here 
\emph{B-specialness}, which can also hold of trees with uncountable branches. By results of Cox and 
Krueger \cite{cox_krueger_indestructible} and Krueger \cite{krueger_sch}, $\IGMP$ follows from the conjunction 
of $\GMP$ and the assertion that all trees of height and size $\omega_1$ are B-special. 
Also, Cox and Krueger prove \cite{cox_krueger_indestructible} that $\IGMP$ is compatible with any 
possible value of the continuum with cofinality at least $\omega_2$. Motivated by the question of 
whether $\IGMP$ is compatible with $\cf(2^\omega) = \omega_1$, we prove the following variation on 
a theorem of Laver \cite{laver_random}, who proved that forcing with a measure algebra over any 
model of Martin's Axiom preserves the fact that all trees of height and size $\omega_1$ with no 
uncountable branches are special.

\begin{thmc}
  If $\PFA$ holds and $\bb{B}$ is a measure algebra, then in $V^{\bb{B}}$ every tree of height and 
  size $\omega_1$ is $B$-special.
\end{thmc}

Theorem C implies that, in the extension of any model of $\PFA$ by a measure algebra, an indestructible 
version of $\neg \wKH$ holds (see Corollary~\ref{wkh_continuum_cor} for a precise statement). In 
particular, this indestructible version of $\neg \wKH$ is compatible with $\cf(2^\omega) = \omega_1$.

In the last section of the paper, we investigate the effect of forcing axioms 
for Suslin and almost Suslin trees on guessing model principles. The axiom $\PFA(S)$ (introduced by Todorcevic in \cite{todorcevic_forcing} and defined more precisely in Section~\ref{pfas_section}) 
is the assertion that $S$ is a Suslin tree and the conclusion of $\PFA$ holds when restricted 
to proper forcings that preserve the fact that $S$ is a Suslin tree.\footnote{Most other sources require 
$S$ to be a \emph{coherent} Suslin tree in the statement of $\PFA(S)$. Since we will not need coherence 
here, we state the axiom in a more general form.} $\PFA(S)[S]$ is the assertion that 
the universe is obtained from forcing over a model of $\PFA(S)$ with the Suslin tree $S$. The axiom $\PFA(T^*)$ (introduced by Krueger in \cite{krueger_forcing_axiom} and again defined more precisely in Section~\ref{pfas_section}), 
is the assertion that $T^*$ is an almost Suslin Aronszajn tree and the conclusion of 
$\PFA$ holds when restricted to proper forcings preserving this fact.
Our main theorem in this section is the following:
%

\begin{thmd}
  \begin{enumerate}
  \item Let $S$ denote a Suslin tree. Then $\PFA(S)$ implies $\GMP$, and 
  $\PFA(S)[S]$ implies $\IGMP$. 
  \item Let $T^*$ denote an almost Suslin Aronszajn tree. Then $\PFA(T^*)$ implies 
  $\IGMP$.  
  \end{enumerate}
\end{thmd}

This theorem answers a number of questions of Cox and Krueger, which we record 
here and explicate more thoroughly in Section~\ref{pfas_section}.

\begin{itemize}
  \item In \cite{cox_krueger_indestructible}, Cox and Krueger ask whether 
  $\IGMP$ implies that the pseudointersection number $\mathfrak{p}$ is 
  greater than $\omega_1$. Theorem D answers this negatively, 
  since, in any model of $\PFA(S)[S]$, we have $\mathfrak{p} = \omega_1$.
  \item In \cite{cox_krueger_indestructible}, Cox and Krueger ask whether 
  $\IGMP$ implies that every tree of height and size $\omega_1$ with no cofinal 
  branches is special. 
  Theorem D answers this negatively, since, in any model of $\PFA(T^*)$, 
  $\IGMP$ holds and $T^*$ is a nonspecial Aronszajn tree.
  \item In \cite{krueger_forcing_axiom}, Krueger asks whether $\PFA(T^*)$ implies 
  $\neg \wKH$. Theorem D answers this positively, since $\GMP$, and hence 
  $\IGMP$, implies $\neg \wKH$.
\end{itemize}

The structure of the remainder of the paper is as follows. In Section~\ref{background_sec}, we 
review some background on $\power_\kappa \lambda$ combinatorics and guessing models. In 
Section~\ref{kurepa_sec}, we investigate the effect of $\neg \wKH$ on cardinal arithmetic, proving 
clauses (2) and (3) of Theorem B. In Section~\ref{isp_p_omega_1_section}, we investigate the effect of 
(weakenings of) $\GMP$ on cardinal arithmetic. Among the central technical results of this section is 
Lemma~\ref{downward_coherence_lemma}, a new lemma about covering matrices. This is then used to prove 
Theorem A and clause (1) of Theorem B. Section~\ref{special_sec} contains the proof of Theorem C, and 
Section~\ref{pfas_section} contains the proof of Theorem D.

\subsection{Notation and terminology}

Our terminology and notation is for the most part standard. We use \cite{jech} as our standard 
background reference for set theory and refer the reader there for any undefined notions or notations.
We record a few notational conventions here at the outset.
If $\kappa < \nu$ are infinite cardinals, with $\kappa$ regular, then $S^\nu_\kappa 
:= \{\alpha < \nu \mid \cf(\alpha) = \kappa\}$. If $\kappa$ is an infinite cardinal and $X$ is a set 
with $|X| \geq \kappa$, then $\power_\kappa X := \{x \subseteq X \mid |x| < \kappa\}$.
If $\P$ is a forcing poset with greatest lower bounds and $p,q$ are compatible conditions in $\P$, 
then $p \wedge q$ denotes their greatest common lower bound.

\section{Background on two-cardinal combinatorics and guessing models} \label{background_sec}

Though much of the previous work motivating this article concerns two-cardinal tree properties, 
we will be working here exclusively with the formulations of these properties in terms 
of guessing models. Since the definitions of the relevant thin and $\mu$-slender 
$(\kappa, \lambda)$-lists and $(\kappa, \lambda)$-trees and the ensuing tree properties 
$\mathsf{(I)TP}(\kappa, \lambda)$ and $\mathsf{(I)SP}(\mu, \kappa, \lambda)$ are somewhat 
involved and will not directly be used in this paper, we refer the reader to the companion 
paper \cite{kurepa_paper} for their precise definition and their connection with the guessing model 
properties studied here and defined below. Before introducing these guessing model properties, though, 
we need some background on two-cardinal combinatorics.

\subsection{Two-cardinal combinatorics}
Temporarily fix a regular uncountable 
cardinal $\kappa$ and a set $X$ with $|X| \geq \kappa$.

\begin{definition}
  Suppose that $\mc C \subseteq \power_\kappa X$.
  \begin{enumerate}
    \item $C$ is \emph{closed} if whenever $D \subseteq C$ is such that $|D| < \kappa$ and 
    $D$ is linearly ordered by $\subseteq$, we have $\bigcup D \in C$;
    \item $C$ is \emph{strongly closed} if whenever $D \subseteq C$ and $|D| < \kappa$, we have 
    $\bigcup D \in C$;
    \item $C$ is \emph{cofinal} if for all $x \in \power_\kappa \lambda$, there is $y \in C$ such 
    that $x \subseteq y$;
    \item $C$ is a \emph{club} in $\power_\kappa X$ if it is closed and cofinal;
    \item $C$ is a \emph{strong club} in $\power_\kappa X$ if it is strongly closed and cofinal.
  \end{enumerate}
  A set $S \subseteq \power_\kappa X$ is \emph{(weakly) stationary} in $\power_\kappa X$ if 
  $S \cap C \neq \emptyset$ for every (strong) club $C \subseteq \power_\kappa X$.
\end{definition}

Given a set $x \subseteq X$ and a function $f: X \rightarrow 
\power_\kappa X$, we say that $x$ is \emph{closed under $f$} if $f(a) \subseteq x$ for 
all $a \in x$. Similarly, if $g:[X]^2 \rightarrow \power_\kappa X$, then $x$ is 
closed under $g$ if $g(a) \subseteq x$ for all $a \in [x]^2$.
The following proposition is immediate.

\begin{proposition}
  Suppose that $f:X \rightarrow \power_\kappa X$ is a function. Then the set 
  $\{x \in \power_\kappa X \mid x \text{ is closed under } f\}$ is a strong club in 
  $\power_\kappa X$. In particular, if $\mc Y \subseteq \power_\kappa X$ is 
  weakly stationary, then there is $x \in \mc Y$ such that $x$ is closed under $f$.
\end{proposition}

The following characterization of the club filter on $\power_\kappa X$ is due to Menas \cite{menas}.

\begin{proposition} \label{menas_prop}
  If $g:[X]^2 \rightarrow \power_\kappa X$ is a function, then the set 
  \[
  C_g := \{x \in \power_\kappa X \mid x \text{ is infinite and closed under } g\}
  \]
  is a club in $\power_\kappa X$. Moreover, for any club $C$ in $\power_\kappa X$, 
  there is $g:[X]^2 \rightarrow \power_\kappa X$ such that $C_g \subseteq C$.
\end{proposition}

\subsection{Guessing models}
We now review the notion of a \emph{guessing model} and the subsequently defined 
\emph{guessing model properties}, which provide an alternative formulation of the 
two-cardinal tree properties of the form $\ISP(\ldots)$.

\begin{definition}
  Suppose that $\theta$ is a regular uncountable cardinal and $M \subseteq H(\theta)$.
  \begin{enumerate}
    \item Given a set $x \in M$, a subset $d \subseteq x$, and a cardinal $\mu$, we say 
    that
    \begin{enumerate}
      \item $d$ is \emph{$(\mu, M)$-approximated} if, for every $z \in M \cap \power_\mu(x)$, 
      there is $e \in M$ such that $d \cap z = e \cap z$;\footnote{In other works, the conclusion of 
      this definition is simply ``$d \cap z \in M$". This is clearly equivalent to what is written here 
      if $M \prec H(\theta)$ (and hence $M$ is closed under intersections). We will want to apply this 
      definition to more general situations, though, and for our purposes this seems like the most 
      appropriate formulation.}
      \item $d$ is \emph{$M$-guessed} if there is $e \in M$ such that $d \cap M = e \cap M$.
    \end{enumerate}
    \item $M$ is a \emph{$\mu$-guessing model for $x$} if every 
    $(\mu, M)$-approximated subset of $x$ is $M$-guessed.
    \item $M$ is a \emph{$\mu$-guessing model} if, for every $x \in M$, it is a $\mu$-guessing 
    model for $x$.
    \item Suppose that $\mu \leq \kappa \leq \theta$ are regular uncountable cardinals
    and $\mc Y \subseteq \power_\kappa H(\theta)$ is stationary. Then 
    $\GMP_{\mc Y}(\mu, \kappa, \theta)$ is the assertion that the set of $M \in \mc Y$ 
    such that $M$ is a $\mu$-guessing models is stationary in $\power_\kappa H(\theta)$.
  \end{enumerate}
\end{definition}

\begin{remark} \label{gmp_convention_rmk}
  In order to cut down on the number of parameters in use and make statements of our 
  results cleaner, we introduce some conventions, all of which are standard in the literature. 
  In the notation $\GMP_{\mc Y}(\mu, \kappa, \theta)$, if $\mc Y$ is omitted, then it 
  should be understood to be $\power_\kappa H(\theta)$. 
  $\GMP(\mu, \kappa, {\geq}\kappa)$ denotes the assertion that 
  $\GMP(\mu, \kappa, \theta)$ holds for all regular $\theta \geq \kappa$. 
  Since the most common first two parameters in $\GMP(\ldots)$ are $\omega_1$ and 
  $\omega_2$, respectively, if $\mc Y \subseteq \power_{\omega_2}H(\theta)$, 
  we let $\GMP_{\mc Y}(\theta)$ denote $\GMP_{\mc Y}(\omega_1, \omega_2, 
  \theta)$ and let $\GMP$ denote the assertion that $\GMP(\theta)$ holds for all 
  regular $\theta \geq \omega_2$.
  By \cite[Propositions 3.2 and 3.3]{viale_weiss}, for a regular cardinal $\kappa \geq \omega_2$,
  $\GMP(\omega_1, \kappa, {\geq}\kappa)$ is equivalent to $\ISP(\omega_1, \kappa, {\geq}\kappa)$, 
  which is typically denoted in the literature as simply $\ISP(\kappa)$ or $\ISP_\kappa$ (we will 
  use the latter in this paper). We note that $\GMP$ follows from the Proper Forcing Axiom 
  \cite[Theorem 4.8]{viale_weiss} and also holds in the extension by the Mitchell forcing 
  $\M(\omega, \kappa)$ if $\kappa$ is supercompact in the ground model 
  \cite[Theorem 5.4]{weiss}.
\end{remark}

We next recall weakenings of $\GMP(\ldots)$ introduced in \cite{kurepa_paper} that 
provide alternative formulations of two-cardinal tree properties of the form 
$\SP(\ldots)$.

\begin{definition}
  Suppose that $\mu \leq \kappa \leq \theta$ are regular uncountable cardinals, 
  $x \in H(\theta)$, $S \subseteq \power_\kappa H(\theta)$, and $M \subseteq H(\theta)$.
  We say that $(M,x)$ is \emph{almost guessed by $S$} if $x \in M$ and, for every 
  $(\mu, M)$-approximated subset $d \subseteq x$, there is $N \in S$ such that
  \begin{itemize}
    \item $x \in N \subseteq M$;
    \item $d$ is $N$-guessed.
  \end{itemize}
  Suppose that $\mc Y \subseteq \power_\kappa H(\theta)$.
  \begin{itemize}
    \item $\AGP_{\mc Y}(\mu, \kappa, \theta)$ is the assertion that, for every 
    $\subseteq$-cofinal $S \subseteq \power_\kappa H(\theta)$ and every $x \in H(\theta)$, 
    the set of $M \in \mc Y$ such that $(M,x)$ is almost guessed by $S$ is stationary in 
    $\power_\kappa H(\theta)$.
    \item $\wAGP_{\mc Y}(\mu, \kappa, \theta)$ is defined in the same way, except that the 
    set of $M$ as in the conclusion is only assumed to be \emph{weakly} stationary in 
    $\power_\kappa H(\theta)$.
  \end{itemize}
\end{definition}

\begin{remark} \label{guessing_remark}
  $\mathsf{(w)AGP}$ stands for ``(weak) almost guessing principle".
  In the notation $\mathsf{(w)AGP}_{\mc Y}(\mu, \kappa, H(\theta))$, we again suppress 
  mention of $\mc Y$ if it is equal to $\power_\kappa H(\theta)$. As with $\GMP$, if 
  $\mc Y \subseteq \power_{\omega_2} H(\theta)$, we let $\mathsf{(w)AGP}_{\mc Y}(\theta)$ 
  denote $\mathsf{(w)AGP}_{\mc Y}(\omega_1, \omega_2, \theta)$.
  
  It is immediate that, 
  for a fixed trio of regular uncountable cardinals $\mu \leq \kappa \leq \theta$ and 
  a stationary $\mc Y \subseteq \power_\kappa H(\theta)$, we have
  \[
    \GMP_{\mc Y}(\mu, \kappa, \theta) \Rightarrow \AGP_{\mc Y}(\mu, \kappa, \theta)
    \Rightarrow \wAGP_{\mc Y}(\mu, \kappa, \theta).
  \]
  To verify the first implication, the key observation is the fact that, for a fixed 
  cofinal $S \subseteq \power_\kappa H(\theta)$, the set
  \[
    \{M \in \power_\kappa H(\theta) \mid \forall a \in M \exists N \in S ~ [a \in N \subseteq 
    M]\}
  \]
  is a club in $\power_\kappa H(\theta)$.
\end{remark}

\section{The weak Kurepa hypothesis and cardinal arithmetic} \label{kurepa_sec}

Recall that, for a regular uncountable cardinal $\mu$, a weak $\mu$-Kurepa tree is a tree of height and 
size $\mu$ with more than $\mu$-many cofinal branches. We use $\wKH(\mu)$ to denote the 
\emph{weak Kurepa hypothesis} at $\mu$, i.e., the assertion that there is a weak $\mu$-Kurepa tree. 
Then $\neg \wKH(\mu)$ is the assertion that every tree of height and size $\mu$ has at most 
$\mu$-many cofinal branches. Note that $\neg \wKH$ entails $2^{<\mu} > \mu$, since otherwise 
${^{<\mu}}2$ is a weak $\mu$-Kurepa tree. We omit mention of $\mu$ and write simply $\wKH$ if $\mu = \omega_1$.

In \cite[Theorem 9.3]{kurepa_paper}, we show that, if $\mu$ is a regular uncountable cardinal, then 
$\neg \wKH(\mu)$ follows from $\wAGP(\mu, \mu^+, \mu^+)$. In particular, $\neg \wKH$ follows from 
$\GMP$ (this conclusion had been proven earlier in \cite[Theorem 2.8]{cox_krueger_indestructible}).

In this section, we address the influence of $\neg \wKH(\mu)$ on the value of
$2^{\mu}$. In particular, we will show that $\neg \wKH(\mu)$ forces $2^{\mu}$ to
be as small as possible relative to $2^{<\mu}$ \emph{if} $2^{<\mu} < \mu^{+\mu}$
but does not have the same influence in general. In particular, $\neg \wKH$ 
implies that $2^{\omega_1} = 2^\omega$ if $2^\omega < \aleph_{\omega_1}$, but not necessarily otherwise.
This should be contrasted with (a weakening of) $\GMP$, which implies $\neg \wKH$ and, as we will see in Section
\ref{isp_p_omega_1_section}, forces $2^{\omega_1}$ to be as small as possible relative
to $2^\omega$ regardless of the value of $2^\omega$.

These arguments make use of the notion of \emph{meeting numbers}, which will also be used in 
the arguments of Section \ref{isp_p_omega_1_section}.

\begin{definition} [\cite{matet_towers_and_clubs}]
  Suppose that $\kappa \leq \lambda$ are infinite cardinals. Then the \emph{meeting number} 
  $m(\kappa, \lambda)$ is the minimal cardinality of a collection $\mc Y \subseteq [\lambda]^\kappa$ 
  such that, for all $x \in [\lambda]^\kappa$, there is $y \in \mc Y$ such that $|x \cap y| = \kappa$.
\end{definition}

\begin{remark}
  By a standard diagonalization argument, if $\kappa < \lambda$ and $\cf(\kappa) =
  \cf(\lambda)$, then $m(\kappa, \lambda) > \lambda$. Therefore, for general
  $\kappa < \lambda$, the smallest possible value $m(\kappa, \lambda)$ can take
  is $\lambda$ if $\cf(\kappa) \neq \cf(\lambda)$ and $\lambda^+$ if
  $\cf(\kappa) = \cf(\lambda)$. As we will see in Section \ref{isp_p_omega_1_section}, 
  a weakening of $\GMP$ forces
  $m(\kappa, \lambda)$ to always attain this minimum possible value.
\end{remark}

One reason for interest in meeting numbers is the fact that they provide a simple 
alternate formulation of Shelah's Strong Hypothesis ($\SSH$), which is the assertion that, 
for every singular cardinal $\lambda$, the pseudopower $\mathrm{pp}(\lambda)$ is 
equal to $\lambda^+$.

\begin{theorem}[Matet, \cite{matet_meeting_numbers}] \label{matet_meeting_thm}
  The following are equivalent:
  \begin{enumerate}
    \item Shelah's Strong Hypothesis;
    \item for every singular cardinal $\lambda$ of countable cofinality, we have 
    $m(\omega, \lambda) = \lambda^+$;
    \item for all infinite cardinals $\kappa < \lambda$, we have $m(\kappa, \lambda) = 
    \lambda^+$ if $\cf(\kappa) = \cf(\lambda)$ and $m(\kappa, \lambda) = \lambda$ 
    if $\cf(\kappa) \neq \cf(\lambda)$.
  \end{enumerate}
\end{theorem}

Note that Theorem \ref{matet_meeting_thm} makes it immediately evident that $\SSH$ is indeed a 
strengthening of $\SCH$, since if $\lambda$ is a singular cardinal of countable cofinality, 
then a straightforward calculation yields $\lambda^\omega = 2^\omega \cdot m(\omega, \lambda)$.

The following basic fact about $m(\kappa, \lambda)$ will be useful. For a proof, see 
\cite[Corollary 2.5]{matet_meeting_numbers}.

\begin{proposition} \label{meeting_prop}
 Suppose that $\kappa < \lambda < \kappa^{+\cf(\kappa)}$. Then $m(\kappa, \lambda) = \lambda$.
\end{proposition}

\begin{lemma} \label{general_wkh_meeting_lemma}
  Suppose that $\mu$ is a regular uncountable cardinal, and assume $\neg \wKH(\mu)$. Then 
  $2^\mu = m(\mu, 2^{<\mu})$.
\end{lemma} 

\begin{proof}
  First note that $m(\mu, 2^{<\mu}) \leq |[2^{<\mu}]^\mu| = 2^\mu$. It thus remains to show that 
  $2^\mu \leq m(\mu, 2^{<\mu})$.
  
  Fix $\mc Y \subseteq [{^{<\mu}}2]^\mu$ such that
  \begin{itemize}
    \item $|\mc Y| = m(\mu, 2^{<\mu})$;
    \item for every $x \in [{^{<\mu}}2]^\mu$, there is $y \in \mc Y$ such that $|x \cap y| = \mu$.
  \end{itemize}
  For each $y \in \mc Y$, let $T_y := \{f \in {^{<\mu}2} \mid \exists g \in y ~ [f \subseteq g]\}$, i.e., 
  $T_y$ is the downward closure of $y$ in the tree ${^{<\mu}}2$. Note that $T_y$ is a tree of 
  cardinality $\mu$ and therefore, by $\neg \wKH(\mu)$, $T_y$ has at most $\mu$-many branches of size 
  $\mu$. We naturally identify branches through ${^{<\mu}}2$ of size $\mu$ with elements of 
  ${^\mu}2$.
  
  \begin{claim}
    For every $h \in {^\mu}2$, there is $y \in \mc Y$ such that $h$ is a branch through $T_y$.
  \end{claim}
  
  \begin{proof}
    Fix $h \in {^\mu}2$, and let $x := \{h \restriction \alpha \mid \alpha < \mu\}$ be the set of 
    proper initial segments of $h$. Then $x \in [{^{<\mu}}2]^\mu$, so we can find $y \in \mc Y$ 
    such that $|x \cap y| = \mu$. Then there are unboundedly many $\alpha < \mu$ such 
    that $h \restriction \alpha \in y$; it follows that $h$ is a branch through $T_y$.
  \end{proof}
  
  Since, for each $y \in \mc Y$, $T_y$ has at most $\mu$-many branches of size $\mu$, and since 
  each element of ${^\mu}2$ is a branch through $T_y$ for some $y \in \mc Y$, it follows that 
  $2^\mu \leq \mu \cdot m(\mu, 2^{<\mu})$. Since $\neg \wKH$ holds, we have $2^{<\mu} > \mu$, and 
  thus $m(\mu, 2^{<\mu}) > \mu$. Therefore, $2^\mu = m(\mu, 2^{<\mu})$.
\end{proof}

The following corollary yields clause (2) of Theorem B as a special case.

\begin{corollary}
  Suppose that $\mu$ is a regular uncountable cardinal, $\neg \wKH(\mu)$ holds, and 
  $2^{<\mu} < \mu^{+\mu}$. Then $2^\mu = 2^{<\mu}$.
\end{corollary}

\begin{proof}
  Since $\neg \wKH(\mu)$ holds, Lemma \ref{general_wkh_meeting_lemma} implies that $2^\mu = 
  m(\mu, 2^{<\mu})$. By $\neg \wKH(\mu)$, we have $2^{<\mu} > \mu$, and by assumption we have 
  $2^{<\mu} < \mu^{+\mu}$. Therefore, Proposition \ref{meeting_prop} implies that 
  $m(\mu, 2^{<\mu}) = 2^{<\mu}$. Altogether, this implies that $2^\mu = 2^{<\mu}$.
\end{proof}

We now prove clause (3) of Theorem B, showing that $\neg \wKH(\mu)$ no longer forces $2^{\mu}$ to be as small as possible 
relative to $2^{<\mu}$
if we allow $2^{<\mu} \geq \mu^{+\mu}$. For concreteness, we focus on the case $\mu = \omega_1$ 
and produce a model
in which $\neg \wKH$ holds, $2^\omega = \aleph_{\omega_1}$, and $2^{\omega_1} =
\aleph_{\omega_1+2}$, but it will be evident how to modify the construction to produce
other configurations.

\begin{theorem}
  If the existence of a supercompact cardinal is consistent, then it is consistent
  that $\neg \wKH$ holds, $2^\omega = \aleph_{\omega_1}$, and
  $2^{\omega_1} > \aleph_{\omega_1+1}$.
\end{theorem}

\begin{proof}
  Let $\lambda$ be a supercompact cardinal. By forcing with the Laver preparation
  followed by the forcing to add $\lambda^{++}$-many Cohen subsets to $\lambda$,
  we may assume that $2^\lambda = \lambda^{++}$. Let $\kappa < \lambda$ be weakly compact.
  We can now force with a Radin-type forcing $\R$ with interleaved collapses such that:
  \begin{itemize}
    \item in $V^{\R}$, we have $\lambda = \kappa^{+\omega_1}$, $\lambda$ is strong limit, and
    $2^\lambda = \lambda^{\omega_1} = \lambda^{++}$;
    \item $(V_{\kappa+1})^{V^{\R}} = (V_{\kappa+1})^V$.
  \end{itemize}
  For details about how to define such a forcing notion, see, for example,
  \cite{davis}. The forcing defined there produces an extension in which
  $\lambda = \aleph_{\omega_1}$,
  but by ensuring that all of the points of the Radin club are chosen above $\kappa$
  and by using an initial interleaved L\'{e}vy collapse of the form
  $\Coll(\kappa^{+3}, < \lambda_0)$, where $\lambda_0$ is the first point in the
  Radin club, we obtain a forcing notion with the desired properties.

  Since $(V_{\kappa+1})^{V^{\R}} = (V_{\kappa+1})^V$, $\kappa$ is still weakly compact in
  $V^{\bb{R}}$. Let $\dot{\bb{M}}$ be an $\bb{R}$-name for the standard
  Mitchell forcing $\M(\omega, \kappa)$. Then, in $V^{\R * \dot{\M}}$, we have:
  \begin{itemize}
    \item $\kappa = \omega_2 = 2^\omega$;
    \item $\lambda = \aleph_{\omega_1}$ is strong limit, and $2^\lambda = \lambda^{\omega_1} =
    \lambda^{++}$;
    \item $\neg \wKH$ (cf.\ \cite{T:wKH}).
  \end{itemize}
  Finally, let $\dot{\C}$ be an $\R * \dot{\M}$-name for the forcing to add
  $\lambda$-many Cohen reals. Then, in $V^{\R * \dot{\M} * \dot{\C}}$, we
  have $2^\omega = \lambda = \aleph_{\omega_1}$ and $2^{\omega_1} =
  (\aleph_{\omega_1})^{\aleph_1} = \aleph_{\omega_1 + 2}$. It remains to verify that 
  $\neg \wKH$ holds in $V^{\R * \dot{\M} * \dot{\C}}$. 
  
  It is proven in \cite[\S 1]{T:wKH} that, after forcing over a ground model $W$ with 
  Mitchell forcing $\M(\omega, \mu)$ with $\mu$ inaccessible, $\neg \wKH$ is preserved by 
  further forcing with a finite-support iteration of length at most $\omega_2$ 
  of c.c.c.\ posets of size $\omega_1$ such that each iterand does not add a new uncountable branch to any 
  tree of height $\omega_1$. In particular, it follows that, in $V^{\R \ast \dot{\M}}$, 
  $\neg \wKH$ is preserved after adding $\omega_2$-many or fewer Cohen reals.
  
  In $V^{\R \ast \dot{\M}}$, let $\dot{T}$ be a $\C$-name for a tree of height and size $\omega_1$. 
  We think of conditions in $\C$ as being finite partial functions from $\lambda$ to $2$. For 
  each $A \subseteq \lambda$, let $\C_A$ be the complete suborder consisting of all elements of $\C$ 
  whose domains are subsets of $A$. Since $\C$ has the c.c.c., we can find $A \in [\lambda]^{\aleph_1}$ 
  such that there is a $\C_A$-name $\dot{T}'$ for which $\Vdash_{\C}``\dot{T} = \dot{T}'"$. 
  Let $T$ be the realization of $\dot{T}'$ in $V^{\R \ast \dot{M} \ast \dot{\C}_{\dot{A}}}$. By 
  the previous paragraph, $\neg \wKH$ holds in that model, so $T$ has at most $\omega_1$-many branches there.
  But $V^{\R \ast \dot{M} \ast \dot{\C}}$ is an extension of $V^{\R \ast \dot{M} \ast \dot{\C}_{\dot{A}}}$ 
  by the forcing to add $\lambda$-many Cohen reals, which cannot add new cofinal branches to a tree of 
  height $\omega_1$. Therefore, $T$ still has at most $\omega_1$-many branches in 
  $V^{\R \ast \dot{M} \ast \dot{\C}}$, so $\neg \wKH$ holds there.
\end{proof}

\section{Guessing models and cardinal arithmetic} \label{isp_p_omega_1_section}

In this section, we analyze the effect of guessing model principles on cardinal arithmetic, focusing 
on meeting numbers and pseudopowers. In the process, 
we will prove, for instance, that a weakening of $\GMP$ implies that $2^{\omega_1}$ is as small as 
possible relative to $2^\omega$ and that a weakening of $\GMP(\kappa, \kappa, \geq \kappa)$ 
implies $\SSH$ above $\kappa$.

We will need to introduce a bit of machinery to obtain these results. We first 
recall the notion of a \emph{covering matrix}, introduced by Viale in his proof
that $\PFA$ implies $\SCH$ \cite{viale_pfa_sch}.

\subsection{Covering matrices}

The terminology associated with covering
matrices is slightly inconsistent across sources; we will follow the terminology
of \cite{sharon_viale} and \cite{clh_covering_ii}.

\begin{definition}
  Let $\theta < \lambda$ be regular cardinals. A \emph{$\theta$-covering matrix
  for $\lambda$} is a matrix $\mathcal{D} = \langle D(i, \beta) \mid i <
  \theta, ~ \beta < \lambda \rangle$ such that:
  \begin{enumerate}
    \item for all $\beta < \lambda$, $\langle D(i, \beta) \mid i < \theta \rangle$
    is a $\subseteq$-increasing sequence and $\bigcup_{i < \theta} D(i, \beta) =
    \beta$;
    \item for all $\beta < \gamma < \lambda$ and $i < \theta$, there is $j <
    \theta$ such that $D(i, \beta) \subseteq D(j, \gamma)$.
  \end{enumerate}
  For such a matrix $\mc D$, let $\beta_{\mc D}$ denote the least ordinal
  $\beta$ such that $\otp(D(i, \gamma)) < \beta$ for all $\gamma < \lambda$
  and $i < \theta$.
\end{definition}

We will be especially interested in covering matrices satisfying certain additional properties.

\begin{definition}
  Suppose that $\theta < \lambda$ are regular cardinals and $\mc D$ is a
  $\theta$-covering matrix for $\lambda$.
  \begin{enumerate}
    \item $\mc D$ is \emph{transitive} if, for all $\beta < \gamma < \lambda$
    and all $i < \theta$, if $\beta \in D(i, \gamma)$, then
    $D(i, \beta) \subseteq D(i, \gamma)$.
    \item $\mc D$ is \emph{uniform} if, for every limit ordinal $\beta < \lambda$,
    there is $i < \theta$ such that $D(i, \beta)$ contains a club in $\beta$.
    \item $\mc D$ is \emph{strongly locally downward coherent} if, for all
    $X \in [\lambda]^{\leq \theta}$, there is $\gamma_X < \lambda$ such that,
    for all $\beta \in [\gamma_X, \lambda)$, there is $i < \theta$ such that,
    for all $j \in [i, \theta)$, $X \cap D(j, \beta) = X \cap D(j, \gamma_X)$.
    \item $\CP(\mc D)$ holds if there is an unbounded $A \subseteq \lambda$
    such that $[A]^\theta$ is covered by $\mc D$, i.e., for all
    $X \in [A]^\theta$, there are $\beta < \lambda$ and $i < \theta$ for which
    $X \subseteq D(i, \beta)$.
  \end{enumerate}
\end{definition}

The key ingredient of Viale's proof that $\PFA$ implies $\SCH$ is
the fact that, for every singular cardinal $\mu$ of countable cofinality,
$\PFA$ implies $\CP(\mc D)$ for every strongly locally downward coherent
$\omega$-covering matrix $\mc D$ for $\mu^+$, together with a lemma
asserting that, for every singular cardinal $\mu > 2^\omega$ of countable
cofinality, there is a strongly locally downward coherent $\omega$-covering
matrix $\mc D$ for $\mu^+$ such that $\beta_{\mc D} = \mu$. The proof of this
latter lemma can essentially be split into two steps. In the first, it is proven
that, for every singular cardinal $\mu$, there is a transitive $\cf(\mu)$-covering
matrix $\mc D$ for $\mu^+$ such that $\beta_{\mc D} = \mu$. In the second, it
is shown that, if $\mu > 2^{\cf(\mu)}$, then every such covering matrix
$\mc D$ is strongly locally downward coherent. We show here that, in this second
step, the requirement of $\mu > 2^{\cf(\mu)}$ can be dropped if we additionally
assume that our covering matrix is uniform. We first note that the proof of the first step
already guarantees the existence of such covering matrices. In fact, we will be 
able to arrange the following strengthening of $\beta_{\mc D} = \mu$ that will be 
useful later in this section, when we address the influence of guessing model 
principles on pseudopowers.

\begin{definition}
  Suppose that $\mu$ is a singular cardinal, $\theta = \cf(\mu)$, and 
  $\vec{\mu} = \langle \mu_i \mid i < \theta \rangle$ is an increasing 
  sequence of regular cardinals converging to $\mu$. If $\mc D = \langle D(i,\beta) 
  \mid i < \theta, ~ \beta < \mu^+ \rangle$ is a $\theta$-covering matrix for 
  $\mu^+$, we say that $\mc D$ \emph{respects} $\vec{\mu}$ if $|D(i,\beta)| < \mu_i$ 
  for all $i < \theta$.
\end{definition}

The following lemma is essentially \cite[Lemma 2.4]{sharon_viale}; our formulation is 
slightly stronger than the cited lemma, but the proof there easily yields our desired 
conclusion.

\begin{lemma}[{\cite[Lemma 2.4]{sharon_viale}}] \label{sharon_viale_lemma}
  Suppose that $\mu$ is a singular cardinal, $\theta = \cf(\mu)$, and $\vec{\mu} = 
  \langle \mu_i \mid i < \theta \rangle$ is an increasing sequence of regular cardinals 
  converging to $\mu$. Then there is a uniform, transitive $\theta$-covering matrix
  for $\mu^+$ that respects $\vec{\mu}$.
\end{lemma}

We now show that covering matrices as in Lemma \ref{sharon_viale_lemma} actually
satisfy a strengthening of strong local downward coherence. Our proof is a
modification of an argument of Shelah from his development of PCF theory.
We first recall the following club-guessing theorem.

\begin{theorem}[{\cite[Ch.\ III, \S 2]{cardinal_arithmetic}}] \label{club_guessing_thm}
  Suppose that $\kappa$ and $\nu$ are regular cardinals and $\kappa^+ < \nu$.
  Then there is a sequence $\langle C_\alpha \mid \alpha \in S^\nu_\kappa \rangle$
  such that:
  \begin{enumerate}
    \item for all $\alpha \in S^\nu_\kappa$, $C_\alpha$ is club in $\alpha$;
    \item for every club $C$ in $\nu$, the set $\{\alpha \in S^\nu_\kappa
    \mid C_\alpha \subseteq C\}$ is stationary in $\nu$.
  \end{enumerate}
\end{theorem}

\begin{lemma} \label{downward_coherence_lemma}
  Suppose that $\mu$ is a singular cardinal, $\theta = \cf(\mu)$, and
  $\mc D$ is a uniform, transitive $\theta$-covering matrix for $\mu^+$.
  Then, for every $X \in [\mu^+]^{<\mu}$, there is $\gamma_X < \mu^+$ such
  that, for all $\beta \in [\gamma_X, \mu^+)$, there is $i < \theta$ such
  that, for all $j \in [i, \theta)$, $X \cap D(j, \beta) =
  X \cap D(j, \gamma_X)$.
\end{lemma}

\begin{proof}
  Fix a set $X \in [\mu^+]^{<\mu}$, and suppose for sake of contradiction
  that there is no $\gamma_X$ as in the conclusion of the lemma. Since
  $\mc D$ is transitive, it follows that, for all $\beta < \gamma < \mu^+$
  and all $i < \theta$, if $\beta \in D(i, \gamma)$, then
  $X \cap D(i, \beta) \subseteq X \cap D(i, \gamma)$. Therefore, by the
  nonexistence of a $\gamma_X$ as in the conclusion of the lemma, for
  every $\beta < \mu^+$ we can find a $\beta' > \beta$ and an unbounded
  set $I_\beta \subseteq \theta$ such that, for all $j \in I_\beta$,
  we have $X \cap D(j, \beta) \subsetneq X \cap D(j, \beta')$. Let $E$ be a club
  in $\mu^+$ such that, for all $\gamma \in E$ and all $\beta < \gamma$,
  we have $\beta' < \gamma$. Then, by the transitivity of $\mc D$,
  for all $\beta < \gamma$, both in $E$,
  and for all but boundedly many $j \in I_\beta$, we have
  $X \cap D(j, \beta) \subsetneq X \cap D(j, \gamma)$.

  Let $\kappa := \max\{|X|^+, \theta^+\}$, and let
  $\nu := \kappa^{++}$. Fix a sequence $\langle C_\alpha \mid
  \alpha \in S^\nu_\kappa \rangle$ satisfying the conclusion of Theorem
  \ref{club_guessing_thm}. Assume without loss of generality that each
  $C_\alpha$ contains only limit ordinals.

  We now construct a strictly increasing, continuous
  sequence $\langle \beta_\xi \mid \xi < \nu \rangle$ of elements of $E$
  as follows. Begin by setting $\beta_0 := \min(E)$. If $\xi < \nu$ is a
  limit ordinal and we have defined $\langle \beta_\eta \mid \eta < \xi
  \rangle$, then we are obliged to set $\beta_\xi := \sup\{\beta_\eta \mid
  \eta < \xi\}$. Finally, suppose that $\xi < \nu$ and we have defined
  $\langle \beta_\eta \mid \eta \leq \xi \rangle$. Let $B_{\xi}$ be the
  set of $\alpha \in S^\nu_\kappa$ for which there is $\beta < \mu^+$ such that,
  for some $i < \theta$, $\{\beta_\eta \mid \eta \in C_\alpha \cap
  (\xi + 1)\} \subseteq D(i, \beta)$. For each $\alpha \in B_\xi$, choose
  a $\beta_{\xi, \alpha}$ witnessing this. Finally, choose $\beta_{\xi + 1}
  \in E \setminus (\beta_\xi + 1)$ large enough so that
  $\beta_{\xi + 1} > \beta_{\xi, \alpha}$
  for all $\alpha \in B_\xi$. This completes the construction of
  $\langle \beta_\xi \mid \xi < \nu \rangle$. We first note the following simple
  claim.

  \begin{claim}
    For all $\xi < \nu$ and all $\alpha \in B_\xi$, there is $j < \theta$
    such that $\{\beta_\eta \mid \eta \in C_\alpha \cap (\xi + 1)\}
    \subseteq D(j, \beta_{\xi + 1})$.
  \end{claim}

  \begin{proof}
    Fix $\xi < \nu$ and $\alpha \in B_\xi$. By construction, we can fix
    an $i < \theta$ such that $\{\beta_\eta \mid \eta \in C_\alpha
    \cap (\xi + 1)\} \subseteq D(i, \beta_{\xi, \alpha})$. Since
    $\beta_{\alpha, \xi} < \beta_{\xi + 1}$, we can fix $j \in [i, \theta)$
    such that $\beta_{\alpha, \xi} \in D(j, \beta_{\xi + 1})$. By the
    definition of covering matrix and the transitivity of $\mc D$, we have
    \[
      D(i, \beta_{\xi, \alpha}) \subseteq D(j, \beta_{\xi, \alpha})
      \subseteq D(j, \beta_{\xi + 1}),
    \]
    so $\{\beta_\eta \mid \eta \in C_\alpha \cap (\xi + 1)\} \subseteq
    D(j, \beta_{\xi + 1})$, as desired.
  \end{proof}

  Let $\gamma := \sup\{\beta_\xi \mid \xi < \nu\}$. Since $\mc D$ is uniform,
  there is a club $C \subseteq \gamma$ and an $i^* < \theta$ such that
  $C \subseteq D(i^*, \gamma)$. Let $\bar{C} := \{\xi < \nu \mid \beta_\xi
  \in C\}$. Since $\langle \beta_\xi \mid \xi < \nu \rangle$ enumerates
  a club in $\gamma$, it follows that $\bar{C}$ is a club in $\nu$.
  We can therefore fix $\alpha \in S^\nu_\kappa$ such that
  $C_\alpha \subseteq \bar{C}$.

  Note that, for all $\xi < \nu$, we have $\alpha \in B_\xi$, since
  \[
    \{\beta_\eta \mid \eta \in C_\alpha \cap (\xi + 1)\} \subseteq C
    \subseteq D(i^*, \gamma).
  \]
  Therefore, for each $\xi \in C_\alpha$, we can fix an $i_\xi < \theta$
  such that $\{\beta_\eta \mid \eta \in C_\alpha \cap (\xi + 1)\}
  \subseteq D(i_\xi, \beta_{\xi + 1})$. Also for each $\xi \in C_\alpha$,
  let $\xi^{\dagger}$ denote $\min(C_\alpha \setminus (\xi + 1))$, and
  fix $j_\xi < \theta$ such that $\beta_{\xi + 1} \in D(j_\xi, \beta_{\xi^{\dagger}})$.
  Let $k_\xi := \max\{i_\xi, j_\xi\}$. Since $\kappa$ is a regular cardinal
  and $\kappa > \theta$, we can find a fixed $k < \theta$ and an unbounded
  set $A \subseteq C_\alpha$ such that $k_\xi = k$ for all $\xi \in A$.

  \begin{claim}
    Suppose that $j \in [k, \theta)$ and that $\eta < \xi$ are both in $A$.
    Then $D(j, \beta_{\eta^\dagger}) \subseteq D(j, \beta_{\xi^\dagger})$.
  \end{claim}

  \begin{proof}
    By the choice of $i_\xi$ and $j_\xi$, we have $\beta_{\eta^\dagger}
    \in D(i_\xi, \beta_{\xi + 1})$ and $\beta_{\xi + 1} \in D(j_\xi,
    \beta_{\xi^\dagger})$. Since $j \geq k = \max\{i_\xi, j_\xi\}$,
    the transitivity
    of $\mc D$ implies that $\beta_{\eta^\dagger} \in D(j, \beta_{\xi^\dagger})$
    and then, through another application, that $D(j, \beta_{\eta^\dagger})
    \subseteq D(j, \beta_{\xi^\dagger})$, as desired.
  \end{proof}

  To ease the notation, let $A^\dagger$ denote $\{\eta^\dagger \mid \eta \in A\}$.
  Note that $A^\dagger$ is an unbounded subset of $C_\alpha$.
  The previous claim then can be reworded to assert that, for all $\eta < \xi$,
  both in $A^\dagger$, and for all $j \in [k, \theta)$, we have
  $D(j, \beta_\eta) \subseteq D(j, \beta_\xi)$. For each $\eta \in A^\dagger$,
  let $\hat{\eta} = \min(A^\dagger \setminus (\eta + 1))$. Since both
  $\beta_\eta$ and $\beta_{\hat{\eta}}$ are in $E$, we can find
  $\ell_\eta \in [k, \theta)$ such that $X \cap D(\ell_\eta, \beta_\eta)
  \subsetneq X \cap D(\ell_\eta, \beta_{\hat{\eta}})$. Again, since
  $\kappa$ is a regular cardinal greater than $\theta$, we can find a
  fixed $\ell$ and an unbounded $A^* \subseteq A^\dagger$ such that
  $\ell_\eta = \ell$ for all $\eta \in A^*$. Now, for all
  $\eta < \xi$, both in $A^*$, we have
  \[
    X \cap D(\ell, \beta_\eta) \subsetneq X \cap D(\ell, \beta_{\hat{\eta}})
    \subseteq X \cap D(\ell, \beta_\xi),
  \]
  so $\langle X \cap D(\ell, \beta_\eta) \mid \eta \in A^* \rangle$ is a
  strictly $\subsetneq$-increasing sequence of subsets of $X$, contradicting
  the fact that $\otp(A^*) = \kappa$ and $\kappa > |X|$.
\end{proof}

\subsection{Guessing models and meeting numbers}
We next show that $\GMP(\kappa, \kappa, \geq \kappa)$ implies that $\CP(\mc D)$ 
holds whenever $\lambda > \kappa$ is a singular cardinal of countable 
cofinality and $\mc D$ is a uniform, transitive $\omega$-covering matrix 
for $\lambda^+$. We will in fact prove something a bit stronger; to state the result 
precisely, we will need the following notion.

\begin{definition}
  Given a set $M$, a subset $x \subseteq M$ is said to be \emph{bounded} in 
  $M$ if there is $z \in M$ such that $x \subseteq z$. Given uncountable 
  cardinals $\nu \leq \mu$, with $\nu$ regular, 
  we say that $M$ is \emph{internally $(\nu, \mu)$-unbounded} 
  if, for all $x \in \power_\nu M$ such that $x$ is bounded in $M$, there is 
  $y \in M$ such that $|y| < \mu$ and $x \subseteq y$.
\end{definition}

\begin{remark}
  Notice that the property of being internally $(\nu, \mu)$-unbounded becomes stronger 
  as $\nu$ increases (for fixed $\mu$) and weaker as $\mu$ increases (for fixed $\nu$).
  In the extant literature, \emph{internally unbounded} typically means 
  $(\omega_1, \omega_1)$-unbounded. Also, for a set $M$ and uncountable cardinals 
  $\nu \leq \mu$ with $\nu$ regular, the following weak form 
  of internal approachability is readily seen to imply that $M$ is 
  $(\nu, \mu)$-unbounded: 
  $M = \bigcup_{\alpha < \nu} M_\alpha$ for some $\subseteq$-increasing 
  sequence $\langle M_\alpha \mid \alpha < \nu \rangle$ such that $M_\alpha \in M$ 
  and $|M_\alpha| < \mu$ for all $\alpha < \nu$. It then immediately follows that, 
  for all regular uncountable cardinals $\nu < \kappa \leq \theta$, the set 
  \[
  \{M \in \power_\kappa H(\theta) \mid M \text{ is } (\nu, \kappa)\text{-internally unbounded}\}
  \]
  is stationary in $\power_\kappa H(\theta)$.
\end{remark}

\begin{proposition} \label{prop_410}
  Suppose that $\nu < \mu$ are infinite cardinals, $\theta$ is a sufficiently large regular cardinal, 
  and $M \prec H(\theta)$ is a $\mu$-guessing model such that ${^{<\nu}}M \subseteq M$ and $\nu \in M$. 
  Then $M$ is internally $(\nu^+, \mu)$-unbounded. In particular, every $\mu$-guessing model is 
  internally $(\omega_1, \mu)$-unbounded.
\end{proposition}

\begin{proof}
  The proof is quite similar to that of \cite[Theorem 1.4]{krueger_sch}.
  Suppose for sake of contradiction that $M$ is not internally $(\nu^+, \mu)$-unbounded, 
  and fix a set $z \in M$ and an $x \in \power_{\nu^+}(z \cap M)$ such that there is no 
  $y \in M$ such that $|y| < \mu$ and $x \subseteq y$. Since ${^{<\nu}}M \subseteq M$, we 
  know that $|x| = \nu$; enumerate it as $\langle a_\beta \mid \beta < \nu \rangle$. For 
  each $\gamma < \nu$, let $b_\gamma := \{a_\beta \mid \beta < \gamma\}$. By the closure of 
  $M$, we have $b_\gamma \in M$. Let $B := \{b_\gamma \mid \gamma < \nu\}$, and note that 
  $B \subseteq \power_\nu z \in M$ and $B \subseteq M$.
  
  We claim that $B$ is $(\mu, M)$-approximated. To this end, fix an arbitrary $w \in M \cap 
  \power_\mu (\power_\nu z)$, and note that $\bigcup w \in \power_\mu z \cap M$. 
  If $w \cap B$ had cardinality $\nu$, then we would have $b_\gamma \in w$ for unboundedly 
  many $\gamma < \nu$, which would then imply that $x \subseteq \bigcup w$, contradicting our 
  choice of $x$. Therefore, $w \cap B$ has cardinality less than $\nu$, so, by the closure of $M$, 
  we have $w \cap B \in M$. 
  
  Since $w$ was arbitrary, it follows that $B$ is $(\mu, M)$-approximated. Therefore, $B$ is 
  $M$-guessed, so we can fix $d \in M$ such that $d \cap M = B$. By elementarity, we have 
  $d \subseteq \power_\nu z$. Moreover, since $|B| = \nu$ and $\nu+1 \subseteq M$, it follows again 
  from elementarity that $|d| = \nu$, and therefore that $d = B$. But then $\bigcup d \in M$ 
  and $\bigcup d = \bigcup B = x$, again contradicting our choice of $x$.
\end{proof}

%

By Proposition~\ref{prop_410} and Remark~\ref{guessing_remark}, the assumption 
of $\wAGP_{\mc Y}(\kappa, \kappa, \lambda^{++})$ in the following theorem 
follows immediately from $\GMP(\kappa, \kappa, \lambda^{++})$ and hence also from the 
stronger $\GMP(\omega_1, \kappa, {\geq}\kappa)$, which is equivalent to $\ISP_\kappa$ 
(cf.\ Remark~\ref{gmp_convention_rmk}).

\begin{theorem} \label{cp_thm}
 Suppose that $\kappa < \lambda$ are uncountable cardinals, with $\kappa$ regular and 
 $\lambda$ singular of countable cofinality. Let 
 \[
 \mc Y := \{M \in \power_\kappa H(\lambda^{++}) \mid M \text{ is } (\omega_1, \lambda)\text{-internally unbounded}\},
 \]
 and suppose that
 $\wAGP_{\mc Y}(\kappa, \kappa, \lambda^{++})$ holds. Then $\CP(\mc D)$ holds 
 for every uniform, transitive $\omega$-covering matrix $\mc D$ for $\lambda^+$.
 Moreover, if $m(\omega, \mu) \leq \lambda^+$ for all $\mu < \lambda$, then $m(\omega, \lambda) = \lambda^+$.
\end{theorem}

\begin{proof}
  Fix a uniform, transitive $\omega$-covering matrix $\mc D$ for $\lambda^+$. 
  For all $\alpha < \beta < \lambda^+$, let $j_{\alpha\beta}$ be the least $j < \omega$ such that 
  $\alpha \in D(j, \beta)$.
  For each $\beta < \lambda^+$ and each $i < \omega$, define a function $g_{\beta, i}: \beta 
  \rightarrow \omega$ by letting $g_{\beta, i}(\alpha) := \max\{j_{\alpha\beta}, i\}$ for all 
  $\alpha < \beta$.
  
  By Lemma \ref{downward_coherence_lemma}, for each $X \in [\lambda^+]^{<\lambda}$, we can fix
  $\gamma_X < \lambda^+$ such that, for all $\beta \in [\gamma_X, \lambda^+)$, there is $i < \omega$ 
  such that, for all $j \in [i, \theta)$, we have $X \cap D(j,\beta) = X \cap D(j, \gamma_X)$.
  For each $X \in [\lambda^+]^{<\lambda}$ and each $i < \omega$, define a function 
  $h_{X,i}:X \rightarrow \omega$ by letting $h_{X,i} := g_{\gamma_X, i} \restriction X$.
  Also, let $\pi_0:\lambda^+ \times \omega \rightarrow \lambda^+$ and $\pi_1:\lambda^+ \times \omega \rightarrow 
  \omega$ be the projection maps. We view functions of the form $g_{\beta, i}$ and 
  $h_{X,i}$ as subsets of $\lambda^+ \times \omega$ in the natural way.
  
  Let $S := \{N \in \power_\kappa H(\lambda^{++}) \mid N \prec H(\lambda^{++}) \text{ and }\cf(\sup(N \cap \lambda^{+})) > \omega\}$. 
  Since $\wAGP_{\mc Y}(\kappa, \kappa, \lambda^{++})$ holds, we can find $M \in \mc Y$ such that
  \begin{itemize}
    \item $\lambda^+ \times \omega \in M$;
    \item for all $y \in M$, we have $y \cap (\lambda^+ \times \omega), \sup(y \cap \lambda^+) \in M$;
    \item for all $y \in M \cap \power_\lambda (\lambda^+ \times \omega)$ and every $i < \omega$, we have 
    $\gamma_{\pi_0[y]}, h_{\pi_0[y], i} \in M$ and, if $|y| < \kappa$, then $y \subseteq M$;
    \item $(M, \lambda^+ \times \omega)$ is almost $\kappa$-guessed by $S$.
  \end{itemize}
  
  Let $\delta:= \sup(M \cap \lambda^+)$.
  
  \begin{claim}
    There is $i < \omega$ such that $g_{\delta, i}$ is $(\kappa, M)$-approximated.
  \end{claim}
  
  \begin{proof}
    Suppose not. Then, for every $i < \omega$, we can fix $y_i \in M \cap \power_\kappa (\lambda^+ 
    \times \omega)$ such that there is no $z \in M$ for which $g_{\delta, i} \cap y_i = 
    z \cap y_i$. Let $y := \bigcup_{i < \omega} y_i$. Since $M \in \mc Y$, we can find 
    $w \in M$ such that $y \subseteq w$ and $|w| < \lambda$. By our choice of $M$, we can assume 
    that $w \subseteq \lambda^+ \times \omega$. Let $X := \pi_0[w]$. Again by our choice of $M$, 
    we have $\gamma_X \in M$ and $h_{X,i} \in M$ for all $i < \omega$. In particular, 
    $\gamma_X < \delta$. We can therefore find $i < \omega$ such that, for all $j \in [i, \omega)$, 
    we have $X \cap D(j, \delta) = X \cap D(j, \gamma_X)$. Unraveling the definitions, this 
    implies that $g_{\delta,i} \restriction X = h_{X,i}$. But then $g_{\delta,i} \cap y_i = 
    h_{X,i} \cap y_i$, and we have $h_{X,i} \in M$, contradicting our choice of 
    $y_i$.
  \end{proof}
  
  Since $(M, \lambda^+ \times \omega)$ is almost $\kappa$-guessed by $S$, we can find $N \in S$ 
  such that $\lambda^+ \times \omega \in N \subseteq M$ and an $e \in N$ such that $e \cap N = 
  g_{\delta,i} \cap N$. By elementarity, $e$ is a function from $\lambda^+$ to $\omega$.
  Let $\gamma:= \sup(N \cap \lambda^+)$. Since $\cf(\gamma) > \omega$, 
  we can find $j \geq i$ such that $\{\alpha \in N \cap \lambda^+ \mid e(\alpha) = j\}$ 
  is unbounded in $\gamma$. Let $H:= \{\alpha < \lambda^+ \mid e(\alpha) = j\}$. By elementarity, 
  $H$ is unbounded in $\lambda^+$. Moreover, for every $x \in [H]^{\aleph_0} \cap N$, we have 
  $x \subseteq D(j,\delta)$, so, by elementarity, 
  \[
    N \models \exists \beta < \lambda^+ ~ (x \subseteq D(j, \beta))
  \]
  But then, by another application of elementarity, we have 
  \[
    H(\lambda^{++}) \models \forall x \in [H]^{\aleph_0} ~ \exists \beta < \lambda^+ ~ (x \subseteq 
    D(j, \beta)).
  \]
  In particular, $H$ witnesses $\CP(\mc D)$.
  
  For the ``moreover" clause, suppose that $m(\omega, \mu) \leq \lambda^+$ for all $\mu < \lambda$. 
  We will show that $m(\omega, \lambda^+) = \lambda^+$. Fix a uniform, transitive $\omega$-covering 
  matrix $\mc D$ for $\lambda^+$ such that $\beta_{\mc D} = \lambda$.
  By hypothesis, for all $i < \omega$ and $\beta < \lambda^+$, we can fix a set 
  $\mc Y(i,\beta) \subseteq [D(i,\beta)]^{\omega}$ such that $|\mc Y(i,\beta)| \leq \lambda^+$ and, for all 
  $x \in [D(i,\beta)]^{\omega}$, there is $y \in \mc Y(i,\beta)$ such that $|x \cap y| = \omega$. 
  Let $H \in [\lambda^+]^{\lambda^+}$ witness $\CP(\mc D)$, let 
  $\pi:H \rightarrow \lambda^+$ be the unique order-preserving bijection, and let 
  \[
    \mc Y := \{\pi[x \cap H] \mid i < \omega, ~ \beta < \lambda^+, ~ x \in \mc Y(i,\beta)\}.
  \]
  Clearly, we have $\mc Y \subseteq [\lambda^+]^{\omega}$ and $|\mc Y| = \lambda^+$. To see that 
  $\mc Y$ witnesses $m(\omega, \lambda^+) = \lambda^+$, fix an arbitrary $x \in [\lambda^+]^\omega$. 
  Let $\bar{x} = \pi^{-1}[x]$. Then $\bar{x} \in [H]^\omega$, so we can fix $i < \omega$ and $\beta < 
  \lambda^+$ such that $\bar{x} \subseteq D(i,\beta)$. By our choice of $\mc Y(i,\beta)$, there is 
  $\bar{y} \in \mc Y(i,\beta)$ such that $|\bar{x} \cap \bar{y}| = \omega$. Then 
  $y := \pi[\bar{y} \cap H]$ is in $\mc Y$ and $|x \cap y| = \omega$, as desired.
\end{proof}

In \cite{cox_krueger_quotients}, Cox and Krueger prove that $\GMP$ is
consistent with arbitrarily large values of the continuum. In particular, their
methods allow for the construction of a model in which $\GMP$ holds
and the continuum is a singular cardinal of cofinality $\omega_1$. In such a model,
we necessarily have $2^\omega < 2^{\omega_1}$, so $\GMP$, in contrast with, 
for instance, $\MA_{\omega_1}$, does not imply that $2^{\omega_1} = 2^\omega$. However, 
as we now show, it turns out
that $\GMP$ does have a significant effect on $2^{\omega_1}$, in fact forcing
it to be as small as possible relative to the value of $2^\omega$.

The following corollary, together with the ensuing remark, yields clause (1) of Theorem B. Similar 
corollaries can readily be obtained about the influence of principles of the form 
$\wAGP_{\mc Y_\lambda}(\mu, \mu^+, \lambda^{++})$ on the relationship between 
$2^\mu$ and $2^{<\mu}$ under appropriate hypotheses about the values of $m(\nu, \mu^+)$ 
for $\nu \leq \mu$, where $\mu$ is an arbitrary regular uncountable cardinal.

\begin{corollary} \label{isp_and_omega_1_cor}
  Suppose that $\wAGP_{\mc Y_\lambda}(\lambda^{++})$ holds for all singular 
    $\lambda$ of countable cofinality, where 
    \[
      \mc Y_\lambda := \{M \in \power_{\omega_2} H(\lambda^{++}) \mid M \text{ is } 
      (\omega_1, \lambda)\text{-internally unbounded}\}.
    \]
    Then, for every uncountable cardinal $\mu$, we have 
    \[
    m(\omega, \mu) = \begin{cases}
      \mu & \text{if } \cf(\mu) \neq \omega \\
      \mu^+ & \text{if } \cf(\mu) = \omega.
      \end{cases}
  \]
    In particular,  
    \[
      2^{\omega_1} = \begin{cases}
        2^\omega & \text{if } \cf(2^\omega) \neq \omega_1 \\
        (2^\omega)^+ & \text{if } \cf(2^\omega) = \omega_1.
      \end{cases}
    \]
\end{corollary}

\begin{proof}
  The proof is by induction on $\mu$. The base case of $\mu = \omega_1$ is trivial.
  Suppose that $\mu = \nu^+$ and we know that $m(\omega, \nu) \leq \mu$. Then, by 
  \cite[Proposition 2.4(v)]{matet_meeting_numbers}, we have $m(\omega, \mu) = 
  \max\{\mu, m(\omega, \nu)\} = \mu$. Suppose next that $\mu$ is a limit cardinal of uncountable 
  cofinality and $m(\omega, \nu) < \mu$ for all $\nu < \mu$. Then 
  \cite[Proposition 2.4(vi)]{matet_meeting_numbers} implies that 
  $m(\omega, \mu) = \sup\{m(\omega, \nu) \mid \nu < \mu\} = \mu$. Finally, suppose 
  that $\mu$ is a singular cardinal of countable cofinality and $m(\omega, \nu) < \mu$ 
  for all $\nu < \mu$. Then the ``moreover" clause of Theorem \ref{cp_thm} implies 
  that $m(\omega, \mu) = \mu^+$.
  
  Now \cite[Theorem 1.1]{matet_meeting_numbers} implies that, for all infinite cardinals 
  $\sigma < \mu$, we have
  \[
    m(\sigma, \mu) = \begin{cases}
      \mu & \text{if } \cf(\mu) \neq \cf(\sigma) \\ 
      \mu^+ & \text{if } \cf(\mu) = \cf(\sigma),
    \end{cases}
  \]
  as desired.
  
  To see the ``in particular" clause, first note that, as proven in \cite[Theorem 9.3]{kurepa_paper}, 
  $\wAGP(\omega_2)$ implies 
  $\neg \wKH$, so, \emph{a fortiori}, the hypothesis of the corollary also implies $\neg \wKH$. 
  By Lemma~\ref{general_wkh_meeting_lemma}, we therefore have $2^{\omega_1} = 
  m(\omega_1, 2^\omega)$, and the conclusion follows.
\end{proof}

\begin{remark}
  As the proof of Corollary~\ref{isp_and_omega_1_cor} makes clear, to obtain the ``in particular" clause, 
  we only require $\neg \wKH$ together with the conclusion of the corollary applied to $m(\omega_1, 2^\omega)$. 
  It therefore suffices to assume $\wAGP(\omega_2)$ together with 
  $\wAGP_{\mc Y_\lambda}(\lambda^{++})$ for all singular $\lambda < 2^\omega$ of 
  countable cofinality, as reflected in the statement of Theorem B.
\end{remark}

\subsection{Pseudopowers}

We end this section by investigating the effect of guessing model principles on pseudopowers, 
leading up to the proof of Theorem A. We first recall some necessary definitions; 
for efficiency, we opt to give 
these definitions in the specific settings in which we will need them rather than in full generality.

\begin{definition}
  Suppose that $\theta$ is an infinite regular cardinal and  $\vec{\mu} = \langle \mu_i \mid i < \theta 
  \rangle$ is a sequence of regular cardinals. 
  We let $\prod \vec{\mu}$ denote $\prod_{i < \theta} \mu_i$. If $f,g \in \prod \vec{\mu}$, then 
  we write $f <^* g$ to indicate that $|\{i < \theta \mid g(i) \leq f(i)\}| < \theta$.
\end{definition}

Recall that, if $\mu$ is a singular cardinal, then $\pp(\mu)$ denotes the \emph{pseudopower} of $\mu$. 
Since we will not need it here, we do not give the precise definition of $\pp(\mu)$, referring the 
reader to \cite{matet_meeting_numbers} for all necessary definitions or to
\cite{cardinal_arithmetic} for a more encyclopedic treatment. 
Instead, we simply note here that we always have $\pp(\mu) 
\geq \mu^+$, and that the following consequence of $\pp(\mu) > \mu^+$ follows immediately from 
\cite[Observation 4.4]{matet_meeting_numbers}.

\begin{lemma} \label{cofinality_lemma}
  Suppose that $\mu$ is a singular cardinal and $\pp(\mu) > \mu^+$. 
  Then there is an increasing sequence of regular cardinals $\vec{\mu} = \langle \mu_i \mid i < \cf(\mu) 
  \rangle$ converging to $\mu$ such that $\cf \left( \prod \vec{\mu}, <^* \right) = \mu^{++}$.
\end{lemma}

We also recall that Shelah's Strong Hypothesis ($\SSH$) is the assertion that $\pp(\mu) = \mu^+$ for 
every singular cardinal $\mu$. Given a cardinal $\kappa$, we say that \emph{$\SSH$ 
holds above $\kappa$} if $\pp(\mu) = \mu^+$ for every singular cardinal $\mu > \kappa$. Recall also 
that the Singular Cardinals Hypothesis ($\SCH$) is the assertion that, for every singular strong limit 
cardinal $\mu$, we have $2^\mu = \mu^+$. As with $\SSH$, we say that $\SCH$ holds above some
cardinal $\kappa$ if $2^\mu = \mu^+$ for every singular strong limit cardinal $\mu > \kappa$.
As noted after Theorem \ref{matet_meeting_thm}, the characterization of $\SSH$ in terms of meeting numbers makes 
it clear that $\SSH$ implies $\SCH$. It is not as immediately obvious that, for an arbitrary cardinal 
$\kappa$, $\SSH$ above $\kappa$ implies $\SCH$ above $\kappa$, since the characterization of 
$\SSH$ above $\kappa$ in terms of meeting numbers is not as clean as that of Theorem \ref{matet_meeting_thm}. 
Nonetheless, it is true, and follows from related results in \cite{matet_meeting_numbers} and 
\cite{cardinal_arithmetic}:

\begin{proposition}
  Let $\kappa$ be an infinite cardinal such that $\SSH$ holds above $\kappa$. Then $\SCH$ holds above $\kappa$.
\end{proposition}

\begin{proof}
  Suppose not, and let $\mu > \kappa$ be the least witness to the failure of $\SCH$ above $\kappa$. 
  In particular, $\mu$ is strong limit, $2^\mu = \mu^{\cf(\mu)} > \mu^+$, and, by Silver's Theorem 
  \cite{silver}, we must have $\cf(\mu) = \omega$. By the discussion after Theorem \ref{matet_meeting_thm} 
  above, we have $\mu^\omega = 2^\omega \cdot m(\omega, \mu) = m(\omega, \mu)$, so we have 
  $m(\omega, \mu) > \mu^+$.
  
  Recall that $\mathrm{cov}(\mu, \mu, \omega_1, 2)$ is the least cardinality of a set 
  $\mc X \subseteq \power_\mu \mu$ such that, for every $a \in \power_{\omega_1} \mu$, there is 
  $b \in \mc X$ such that $a \subseteq b$. By a straightforward argument (cf.\ 
  \cite[Proposition 2.4(viii)]{matet_meeting_numbers}), we have
  \[
    m(\omega, \mu) \leq \max\{\mathrm{cov}(\mu, \mu, \omega_1, 2), \sup\{m(\omega, \chi) \mid \chi < \mu\}\}.
  \]
  Since $\mu$ is strong limit, we have $m(\omega, \chi) < \mu$ for all $\chi < \mu$, and hence 
  $m(\omega, \mu) \leq \mathrm{cov}(\mu, \mu, \omega_1, 2)$. Now note that
  \begin{itemize}
    \item $\pp(\chi) < \mu$ for every singular $\chi < \mu$ (because $\mu$ is strong limit and 
    $\pp(\mu) \leq 2^\mu$ for all $\mu$); and
    \item for all $\chi \in (\kappa, \mu)$, if $\chi$ is a singular cardinal of cofinality $\omega_1$, then 
    $\pp(\chi) = \chi^+$ (because $\SSH$ holds above $\kappa$).
  \end{itemize}
  Therefore, \cite[\S IX, Conclusion 1.8]{cardinal_arithmetic} implies that $\pp(\mu) = 
  \mathrm{cov}(\mu, \mu, \omega_1, 2)$. Since $\mathrm{cov}(\mu, \mu, \omega_1, 2) \geq 
  m(\omega, \mu) > \mu^+$, this contradicts the assumption that $\SSH$ holds above $\kappa$.
\end{proof}

We now prove the main result of this subsection, indicating the impact of instances of $\CP(\mathcal{D})$ 
on values of the form $\cf\left(\prod \vec{\mu}, <^* \right)$.

\begin{theorem} \label{cp_cf_thm}
  Suppose that $\mu$ is a singular cardinal, $\theta = \cf(\mu)$, and 
  $\vec{\mu} = \langle \mu_i \mid i < \theta \rangle$ is an increasing sequence 
  of regular cardinals converging to $\mu$. Suppose moreover that there is a 
  $\theta$-covering matrix for $\mu^+$, $\mc D = \langle D(i,\beta) \mid i < \theta, ~ 
  \beta < \mu^+ \rangle$, such that $\mc D$ respects $\vec{\mu}$ and 
  $\CP(\mc D)$ holds. Then $\cf\left( \prod \vec{\mu}, <^*\right) 
  = \mu^+$.
\end{theorem}

\begin{proof}
  Let $\mc D$ be as in the statement of the proposition, and let $A \subseteq \mu^+$ 
  witness $\CP(\mc D)$. Let $\langle \gamma_\eta \mid \eta < \mu \rangle$ enumerate 
  the first $\mu$-many elements of $A$ in increasing order and, for each $i < 
  \theta$, let $\delta_i:= \sup\{\gamma_\eta \mid \eta < \mu_i\}$. Note that 
  $\cf(\delta_i) = \mu_i$. For each 
  $\beta < \mu^+$, define a function $g_\beta \in \prod \vec{\mu}$ as 
  follows: for each $i < \theta$, let $g_\beta(i)$ be the least 
  $\eta < \mu_i$ such that $\sup(D(i,\beta) \cap \delta_i) < \gamma_\eta$. This is 
  well-defined, since $\mc D$ respects $\vec{\mu}$ and therefore 
  $|D(i,\beta)| < \mu_i = \cf(\delta_i)$.
  
  We claim that $\{g_\beta \mid \beta < \mu^+\}$ is cofinal in $\left(\prod \vec{\mu}, 
  <^*\right)$. To this end, fix $f \in \prod \vec{\mu}$. Let 
  $x = \{\gamma_{f(i)} \mid i < \theta\}$. Since $x \in [A]^{\leq \theta}$ and 
  $A$ witnesses $\CP(\mc D)$, we can fix $i < \theta$ and $\beta < \mu^+$ such 
  that $x \subseteq D(i,\beta)$, and hence $x \subseteq D(j,\beta)$ for all 
  $j \in [i,\theta)$. In particular, for every $j \in [i, \theta)$, we have 
  $\gamma_{f(j)} \in D(j,\beta) \cap \delta_j$, and hence $f(j) < g_\beta(j)$. 
  Therefore, $f <^* g_\beta$, as desired.
\end{proof}

We obtain Theorem A as a corollary.

\begin{proof}[Proof of Theorem A]
  Suppose that $\wAGP_{\mc Y_\lambda}(\kappa, \kappa, \lambda^{++})$ holds for all 
  singular $\lambda > \kappa$ of countable cofinality.
  By \cite[\S 2, Claim 2.4]{cardinal_arithmetic}, if $\mu$ is a singular cardinal of 
  uncountable cardinality and the set of singular cardinals $\nu < \mu$ for which 
  $\pp(\nu) = \nu^+$ is stationary in $\mu$, then $\pp(\mu) = \mu^+$. Therefore, to 
  establish $\SSH$ above $\kappa$, it suffices to show that $\pp(\mu) = \mu^+$ for 
  all singular $\mu> \kappa$ of countable cofinality.
  
  Fix such a $\mu$. By Lemma \ref{cofinality_lemma}, it suffices to show that 
  $\cf\left(\prod \vec{\mu}, <^*\right) = \mu^+$ for every increasing sequence of regular cardinals 
  $\vec{\mu} = \langle \mu_i \mid i < \omega \rangle$ converging to $\mu$. 
  Fix such a sequence $\vec{\mu}$. By Lemma \ref{sharon_viale_lemma}, we can fix a 
  uniform, transitive $\omega$-covering matrix $\mc D$ for $\mu^+$ that respects $\vec{\mu}$.
  By Theorem \ref{cp_thm}, $\CP(\mc D)$ holds, and then, by Theorem \ref{cp_cf_thm}, 
  we have $\cf\left(\prod \vec{\mu}, <^*\right) = \mu^+$, as desired.
\end{proof}

\section{Special trees and random reals} \label{special_sec}

In this section, we take a slight detour to prove a variation of a theorem of 
Laver \cite{laver_random} concerning special trees in random extensions of 
models of forcing axioms. At the end of the section, we will rejoin our main 
narrative path by connecting this result with Cox and Krueger's \emph{indestructible 
guessing model principle} \cite{cox_krueger_indestructible}, a strengthening of 
$\GMP$ that will also provide part of the motivation for the results in 
Section~\ref{pfas_section}. 

Recall that a tree $T$ of height $\omega_1$ is \emph{special} if there is a function $f:T \rightarrow 
\omega$ such that, for all $s,t \in T$, if $s <_T t$, then $f(s) \neq f(t)$. It is immediate that a 
special tree cannot have an uncountable branch. Baumgartner introduced a generalization of this 
notion of specialness that can also be satisfied by trees of height $\omega_1$ that have 
uncountable branches; this notion was used to prove that $\PFA$ implies $\neg\wKH$. In order to 
avoid confusion with the more familiar notion of specialness, we will call Baumgartner's generalization 
\emph{$B$-specialness}.

\begin{definition}[{\cite[\S 7]{baumgartner_pfa}}]
  Suppose that $T$ is a tree of height $\omega_1$. We say that $T$ is \emph{$B$-special} if there is 
  a function $f:T \rightarrow \omega$ such that, for all $s,t,u \in T$, if $f(s) = f(t) = f(u)$ and 
  $s <_T t,u$, then $t$ and $u$ are $<_T$-comparable.
\end{definition}

This can indeed be seen as a generalization of the notion of specialness, since, if $T$ is a tree 
of height $\omega_1$ with no uncountable branches, then $T$ is special if and only if $T$ is 
$B$-special. It is well-known that $\MA_{\omega_1}$ implies that every tree of height and size 
$\omega_1$ with no uncountable branch is special. An elaboration of this argument, also due 
to Baumgartner \cite[Theorem 7.10]{baumgartner_pfa} shows that $\PFA$ implies that every tree of 
height and size $\omega_1$ is $B$-special.

In \cite{laver_random}, Laver proves that, if one forces with a measure algebra over any model of 
$\MA_{\omega_1}$, then, in the resulting forcing extension, it remains true that every tree of 
height and size $\omega_1$ with no uncountable branch is special. In particular, this resolved positively 
the question of whether Suslin's Hypothesis is consistent with $\cf(2^\omega) = \omega_1$. 
In this section, we modify Laver's argument to prove an analogous result indicating that, if one forces 
with a measure algebra over any model of $\PFA$, then, in the resulting forcing extension, every 
tree of height and size $\omega_1$ is $B$-special.

The following proposition will be useful.

\begin{proposition}[{\cite[Proposition 4.3]{cox_krueger_indestructible}}]
\label{b_special_outer_model_prop}
  Suppose that $T$ is a $B$-special tree of height $\omega_1$, and suppose that $W$ is an outer model 
  of $V$ with $(\omega_1)^W = (\omega_1)^V$. Then every uncountable branch of $T$ that is in $W$ is 
  also in $V$.
\end{proposition}

We will also need the following definition and lemma.

\begin{definition}
  Suppose that $T$ is a tree and $B$ is a set of cofinal branches of $T$. A function
  $g:B \rightarrow T$ is called a \emph{Baumgartner function} if $g$ is injective and
  \begin{enumerate}
    \item for all $b \in B$, we have $g(b) \in b$;
    \item for all $b, b' \in B$, if $g(b) < g(b')$, then $g(b') \notin b$.
  \end{enumerate}
\end{definition}

If $B$ is small, then a Baumgartner function with domain $B$ always exists:

\begin{lemma}[{\cite[Lemma 7.6]{baumgartner_pfa}}] \label{baumgartner_function_lemma}
  Suppose that $\kappa$ is a regular cardinal, $T$ is a tree of height $\kappa$,
  $B$ is a set of cofinal branches of $T$, and $|B| \leq \kappa$. Then there is a
  Baumgartner function $g:B \rightarrow T$.
\end{lemma}

\begin{theorem} \label{random_wkh_theorem}
  Suppose that $\PFA$ holds, $\kappa$ is an infinite cardinal, and $\bb{B}$ is the measure algebra 
  on $2^\kappa$, with associated measure $\mu$. Then, in $V^{\bb{B}}$, every tree of height and size 
  $\omega_1$ is $B$-special.
\end{theorem}

\begin{proof}
  For notational simplicity, assume that $\kappa \geq \omega_2$; the same ideas will 
  work if $\kappa \in \{\omega, \omega_1\}$.
  Fix a $\bb{B}$-name $\dot{T}$ for a tree of height and size $\omega_1$. Note that, since 
  $\Vdash_{\bb{B}} ``2^\omega = 2^{\omega_1} = \kappa"$, we have $\Vdash_{\bb{B}} ``\dot{T} 
  \text{ has at most } \kappa\text{-many uncountable branches}"$. Let $\bb{C} = \mathrm{Coll}(\omega_1, 
  \kappa)$ (as defined in $V$). Since $\bb{B}$ has the c.c.c.\ and $\bb{C}$ is $\omega_1$-closed, 
  \cite[Lemma 6]{UNGER:1} implies that every uncountable branch of $\dot{T}$ in the extension by 
  $\bb{B} \times \bb{C}$ is already in the extension by $\bb{B}$. In particular, in $V^{\bb{C}}$, 
  we have $\Vdash_{\bb{B}} ``\dot{T} \text{ has at most } \omega_1\text{-many uncountable 
  branches}"$. Therefore, working in $V^{\bb{C}}$ and letting $\dot{B}$ be a $\bb{B}$-name for 
  the set of all uncountable branches through $\dot{T}$, Lemma \ref{baumgartner_function_lemma} implies 
  that we can find a $\bb{B}$-name $\dot{g}:\dot{B} \rightarrow \dot{T}$ that is forced to be a Baumgartner 
  function.
  
  Still working in $V^{\bb{C}}$, let $\dot{T}_0$ be a $\bb{B}$-name for the set 
  \[
    \{t \in \dot{T} \mid \exists b \in \dot{B} ~ \dot{g}(b) <_{\dot{T}} t \in \dot{b}\},
  \]
  and let $\dot{T}_1$ be a $\bb{B}$-name for $\dot{T} \setminus \dot{T}_0$. Then $\dot{T}_1$ 
  is forced to be a subtree of $\dot{T}$ (with the induced ordering), and, since a tail of every element 
  of $\dot{B}$ is forced to be in $\dot{T}_0$, we have $\Vdash_{\bb{B}} ``\dot{T}_1 \text{ has no 
  uncountable branch}"$.
  
  In $V$, $\bb{C}$ is $\omega_1$-closed and therefore does not add any new codes for Borel subsets of 
  $2^\kappa$. As a result, in $V^\bb{C}$, the measure algebra on $2^\kappa$ is isomorphic to 
  $\bb{B}$, so, when working in $V^{\bb{C}}$, we can still think of $\bb{B}$ as being the 
  measure algebra on $2^\kappa$. Therefore, the proof of the main result of \cite{laver_random} 
  implies that, in $V^{\bb{C}}$, there is a c.c.c.\ forcing poset $\bb{P}$ that adds a $\bb{B}$-name for 
  a specializing function for $\dot{T}_1$. Let us recall how this poset is defined.
  
  Work in $V^{\bb{C}}$. Without loss of generality, we may assume that the underlying set of 
  $\dot{T}$ is forced to be $\omega_1$, and therefore the underlying sets of $\dot{T}_0$ and $\dot{T}_1$ 
  are forced to be subsets of $\omega_1$.
  We can also assume that $\Vdash_{\bb{B}} ``\forall \alpha, \beta < \omega_1 (\alpha <_{\dot{T}} 
  \beta \Rightarrow \alpha < \beta)"$.
  
  For all $\alpha < \omega_1$, let $A_\alpha \subseteq \bb{B}$ be a maximal antichain of conditions 
  deciding the truth value of the statement $``\alpha \in \dot{T}_1"$. For all $\alpha < \beta < \omega_1$, 
  let $A'_{\alpha\beta}$ be a maximal antichain of conditions $b \in \bb{B}$ such that one of the following 
  holds:
  \begin{itemize}
    \item $b$ forces at least one of $\alpha$ and $\beta$ to be in $\dot{T}_0$; or 
    \item $b$ forces both $\alpha$ and $\beta$ to be in $\dot{T}_1$ and $b$ decides the truth value 
    of the statement $``\alpha <_{\dot{T}_1} \beta"$.
  \end{itemize}
  
  For all $\gamma < \omega_1$, let $\bb{B}_\gamma$ be the complete subalgebra of $\bb{B}$ generated by 
  $\bigcup_{\alpha \leq \gamma} A_\alpha \cup \bigcup_{\alpha < \beta \leq \gamma} A'_{\alpha\beta}$. 
  Since $\bb{B}$ has the c.c.c., each $\bb{B}_\gamma$ is countably generated and is thus 
  isomorphic to the measure algebra on $2^\omega$. We can therefore fix a 
  countable subset $\mathbb{B}^*_\gamma \subseteq \mathbb{B}_\gamma$ such that, for all 
  $b \in \bb{B}_\gamma$ and all $\varepsilon > 0$, there is $b^* \in \bb{B}^*_\gamma$ such that 
  $\mu(b^* - b) < \varepsilon \cdot \mu(b^*)$.
    
  We now define our poset $\bb{P}$. The conditions in $\bb{P}$ are all functions $p$ with domains of 
  the form $E_p \times W_p$, where $E_p \in [\omega_1]^{<\omega}$ and $W_p \in [\omega]^{<\omega}$, 
  such that
  \begin{itemize}
    \item for all $(\alpha, n) \in E_p \times W_p$, either $p(\alpha, n) = 0$ or $p(\alpha, n) \in 
    \bb{B}_\alpha$ and $\mu(p(\alpha, n)) > \frac{1}{2}$;
    \item for all $\alpha < \beta$ in $E_p$ and all $n \in W_p$, either $p(\alpha, n) \wedge 
    p(\beta, n) = 0$ or $p(\alpha, n) \wedge p(\beta, n) \Vdash_{\bb{B}} ``\alpha \not<_{\dot{T}_1} 
    \beta"$.
  \end{itemize}
  Note that, for a condition $b \in \bb{B}$ and ordinals $\alpha < \beta < \omega_1$, if we 
  write $b \Vdash_{\bb{B}}``\alpha <_{\dot{T}_1} \beta"$, then implicit in this assertion is the 
  fact that $b$ forces both $\alpha$ and $\beta$ to be in $\dot{T}_1$. Therefore, one of the ways 
  in which we could have $p(\alpha, n) \wedge p(\beta, n) \Vdash_{\bb{B}} ``\alpha \not<_{\dot{T}_1} 
  \beta"$ in the above bullet point is for $p(\alpha, n) \wedge p(\beta, n)$ to force either 
  $\alpha$ or $\beta$ to be in $\dot{T}_0$.
  
  If $p,q \in \bb{P}$, then we let $q \leq_{\bb{P}} p$ if and only if $E_q \supseteq E_p$, $W_q \supseteq 
  W_p$, $q(\alpha, n) \leq_{\bb{B}} p(\alpha, n)$ for all $(\alpha, n) \in E_p \times W_p$, and, 
  for all such $(\alpha, n)$, if $p(\alpha, n) > 0$, then $q(\alpha, n) > 0$.
  
  By \cite[Lemma 3]{laver_random}, $\bb{P}$ has the c.c.c.\ in $V^{\bb{C}}$. (Our poset $\bb{P}$ has some 
  minor cosmetic differences from the poset considered in \cite{laver_random} due to the fact that we 
  are adding a name for a specializing function for a \emph{subtree} of $\dot{T}$ rather than the 
  entire tree, but all of the arguments of \cite{laver_random} go through without change to our setting.)
  In $V$, let $\dot{\bb{P}}$ be a $\bb{C}$-name for $\bb{P}$. Then $\bb{C} * \dot{\bb{P}}$ is proper. 
  Note that, in $V^{\bb{C}}$, every condition of $\bb{P}$ is in fact in $V$, since it is a function from 
  a finite set of pairs of ordinals into $\bb{B}$. Therefore, in $V$, the set of conditions $(c, \dot{p})$ 
  such that there is a function $p \in V$ for which $c \Vdash_{\bb{C}}``\dot{p} = p"$ is dense in 
  $\bb{C} * \dot{\bb{P}}$, so we will assume that we are working exclusively with such conditions 
  and will write $(c,p)$ instead of $(c, \dot{p})$.
  
  We now isolate a collection of $\omega_1$-many dense subsets of $\bb{C} \ast \dot{\bb{P}}$ to which we 
  will apply $\PFA$. First note that, in $V^{\bb{C} \times \bb{B}}$, for all $\alpha \in T_0$, 
  the fact that $g$ is a Baumgartner function implies that there is a unique uncountable branch 
  $b$ of $T$ such that $g(b) <_T \alpha \in b$. Denote the value $g(b)$ for this unique branch $b$ 
  by $\eta_\alpha$, and note that $\eta_\alpha \in T_1$. If $\alpha \in T_1$, then let 
  $\eta_\alpha = 0$. In $V$, let $\dot{\eta}_\alpha$ be a $\bb{C} \times \bb{B}$-name for $\eta_\alpha$.
  
  For each $\alpha < \omega_1$, let $D_\alpha$ be the set of $c \in \bb{C}$ 
  for which there exists a maximal antichain $A(c, \alpha)$ of $\bb{B}$ such that, for every 
  $b \in A(c, \alpha)$, $(c, b) \in \bb{C} \times \bb{B}$ decides the truth value of the 
  statement $``\alpha \in \dot{T}_1"$ and decides the value of $\dot{\eta}_\alpha$. Since 
  $\bb{C}$ is $\omega_1$-closed and $\bb{B}$ has the c.c.c., $D_\alpha$ is dense in $\bb{C}$. 
  Let $D^*_\alpha$ be the set of $(c,p) \in \bb{C} \ast \dot{\bb{P}}$ such that $c \in D_\alpha$. 
  Then $D^*_\alpha$ is dense in $\bb{C} \ast \dot{\bb{P}}$.
  
  Recall that, in $V^{\bb{C}}$, for every $\gamma < \omega_1$, we constructed $\bb{B}_\gamma$ 
  to be a countably generated complete subalgebra of $\bb{B}$ and also specified a countable 
  subset $\mathbb{B}^*_\gamma \subseteq \bb{B}_\gamma$ such that, for all $b \in \bb{B}_\gamma$ 
  and every $\varepsilon > 0$, there is $b^* \in \bb{B}^*_\gamma$ such that $\mu(b^* - b) < 
  \varepsilon \cdot \mu(b^*)$. 
  In $V$, let $\dot{h}_\gamma$ be a $\bb{C}$-name for a bijection from $\omega$ to 
  $\bb{B}^*_\gamma$, and let $\dot{\bb{A}}_\gamma$ be a $\bb{C}$-name for a countable generating 
  set for $\bb{B}_\gamma$. 
  Since $\bb{C}$ is $\omega_1$-closed, $\dot{h}_\gamma$ and $\dot{\bb{A}}_\gamma$ are forced to be elements of 
  the ground model. For all $\gamma < \omega_1$ and $k < \omega$, let $E^*_{\gamma, k}$ be the 
  set of $(c,p) \in \bb{C} \ast \dot{\bb{P}}$ such that 
  \begin{itemize}
    \item $c$ decides the value of $\dot{h}_\gamma$, say as some function $h_\gamma : \omega \rightarrow 
    \bb{B}$;
    \item $c$ decides the value of $\dot{\bb{A}}_\gamma$;
    \item there is an $n < \omega$ such that $(\gamma, n) \in \dom{p}$, 
    $\mu(p(\gamma, n)) > \frac{1}{2}$, and $\mu(p(\gamma, n) - h_\gamma(k)) < \frac{1}{2} - 
    \frac{1}{2}\mu(h_\gamma(k))$.
  \end{itemize}
  We claim that $E^*_{\gamma, k}$ is dense in $\bb{C} * \dot{\bb{P}}$. To see this fix an arbitrary 
  $(c_0, p_0) \in \bb{C} * \dot{\bb{P}}$. First find a condition $c_1 \leq_{\bb{C}} c_0$, a 
  function $h_\gamma : \omega \rightarrow \bb{B}$, and a set $\bb{A}_\gamma$ such that $c_1 \Vdash_{\bb{C}}``\dot{h}_\gamma 
  = h_\gamma \text{ and } \dot{\bb{A}}_\gamma = \bb{A}_\gamma"$. Then, fix an $n \in \omega \setminus W_{p_0}$. Let $\dot{p}_1$ be a $\bb{C}$-name 
  for a condition in $\dot{\bb{P}}$ that is forced by $c_1$ to have the following properties:
  \begin{itemize}
    \item $E_{\dot{p}_1} = E_{p_0} \cup \{\gamma\}$;
    \item $W_{\dot{p}_1} = W_{p_0} \cup \{n\}$;
    \item for all $(\alpha, m) \in E_{p_0} \times W_{p_0}$, $\dot{p}_1(\alpha, m) = 
    p_0(\alpha, m)$;
    \item for all $(\alpha, m) \in (E_{\dot{p}_1} \times W_{\dot{p}_1}) \setminus 
    (E_{p_0} \times W_{p_0})$, if $(\alpha, m) \neq (\gamma, n)$, then $\dot{p}_1(\alpha, m) = 0$;
    \item $\dot{p}_1(\gamma, n)$ is an element $b$ of $\bb{B}_\gamma$ such that
    \begin{itemize}
      \item[$\circ$] $h_\gamma(k) \leq b$;
      \item[$\circ$] $\frac{1}{2} < \mu(b) < \frac{1}{2} + \frac{1}{2}\mu(h_\gamma(k))$.
    \end{itemize}
  \end{itemize}  
  It is possible to satisfy the final requirement above due to the fact that $\bb{B}_\gamma$ is 
  forced to be atomless. Finally, find $c_2 \leq_{\bb{C}} c_1$ deciding the value of $\dot{p}_1$ 
  as some function 
  $p_1 \in V$. Then $(c_2, p_1) \in E^*_{\gamma, k}$ and $(c_2, p_1) \leq_{\bb{C} * \dot{\bb{P}}} 
  (c_0, p_0)$.
  
  Now apply $\PFA$ to find, in $V$, a filter $G \subseteq \bb{C} * \dot{\bb{P}}$ that meets 
  $D^*_\alpha$ for all $\alpha < \omega_1$ and $E^*_{\gamma, k}$ for all $\gamma < \omega_1$ and 
  $k < \omega$. Also, let $H \subseteq \bb{B}$ be a $V$-generic filter. Working in $V[H]$, let 
  $T$ denote the realization of $\dot{T}$. We will use $G$ and $H$ to construct a function 
  $f:T \rightarrow \omega$ witnessing that $T$ is $B$-special.
  
  We first define subtrees $T^*_0, T^*_1 \subseteq T$ as follows. For every $\alpha < \omega_1$, 
  find a condition $(c^*_\alpha,p^*_\alpha) \in G \cap D^*_\alpha$. Then, in $V$, $A(c^*_\alpha, \alpha)$ 
  was a maximal antichain in $\bb{B}$, so there is a unique $b^*_\alpha \in H \cap A(c^*_\alpha, \alpha)$. 
  Moreover, we know that $(c^*_\alpha, b^*_\alpha)$ decides the truth value of the statement 
  $``\alpha \in \dot{T}_1"$. If $(c^*_\alpha, b^*_\alpha)$ forces $\alpha$ to be in $\dot{T}_1$, 
  then put $\alpha$ into $T^*_1$; otherwise, put $\alpha$ into $T^*_0$. In addition, 
  $(c^*_\alpha, b^*_\alpha)$ decides the value of $\dot{\eta}_\alpha$; let $\eta^*_\alpha$ be this 
  value. In $V$, let $\dot{T}^*_0$ and $\dot{T}^*_1$ be $\bb{B}$-names for $T^*_0$ and $T^*_1$, 
  respectively.
  
  \begin{claim} \label{eta_claim}
    For all $\alpha < \beta \in T^*_0$, if $\alpha \not<_T \beta$, then either 
    $\eta^*_\beta \not<_T \alpha$ or $\eta^*_\alpha \not<_T \beta$.
  \end{claim}
  
  \begin{proof}
    Fix $\alpha < \beta \in T^*_0$, and suppose for sake of contradiction that $\alpha \not<_T \beta$ 
    but $\eta^*_\beta <_T \alpha$ and $\eta^*_\alpha <_T \beta$. Let $c^*$ be the greatest lower bound 
    of $c^*_\alpha$ and $c^*_\beta$ in $\bb{C}$, and let $b^* \leq b^*_\alpha \wedge b^*_\beta$ be 
    a condition in $H$ forcing that $\alpha \not<_T \beta$, $\eta^*_\beta <_T \alpha$, and 
    $\eta^*_\alpha <_T \beta$. 
    Then, in $V$, $(c^*, b^*) \Vdash_{\bb{C} \times \bb{B}}``\dot{\eta}_\alpha = \eta^*_\alpha \text{ and } 
    \dot{\eta}_\beta = \eta^*_\alpha$. In particular, there are $\bb{B}$-names $\dot{d}_\alpha, 
    \dot{d}_\beta$ for uncountable branches through $\dot{T}$ such that 
    \[
    (c^*, b^*) \Vdash_{\bb{C} \times 
    \bb{B}} ``\alpha \in \dot{d}_\alpha, ~ \beta \in \dot{d}_\beta, ~ \dot{g}(\dot{d}_\alpha) = 
    \eta^*_\alpha, ~ \dot{g}(\dot{d}_\beta) = \eta^*_\beta".
    \]
    Since $b^*$ forces $\alpha \not<_T \beta$, it must be the case that $(c^*, b^*)$ forces 
    $\dot{d}_\alpha$ and $\dot{d}_\beta$ to be distinct branches of $T$, and hence 
    $\eta^*_\alpha \neq \eta^*_\beta$. Since $(c^*, b^*)$ forces $\eta^*_\alpha, \eta^*_\beta <_T \alpha$, 
    it forces $\eta^*_\alpha$ and $\eta^*_\beta$ to be $<_T$-comparable. Suppose that 
    $\eta^*_\alpha < \eta^*_\beta$, in which case $(c^*, b^*) \Vdash_{\bb{C} \times \bb{B}} 
    ``\eta^*_\alpha <_T \eta^*_\beta"$. Since $\dot{g}$ is forced to be a Baumgartner function, 
    it must be the case that $(c^*, b^*) \Vdash_{\bb{C} \times \bb{B}} ``\eta^*_\beta \notin 
    \dot{d}_\alpha"$. However, $(c^*, b^*) \Vdash_{\bb{C} \times \bb{B}} ``\eta^*_\beta <_{\dot{T}} 
    \alpha \in \dot{d}_\alpha"$, which is a contradiction. A symmetric argument yields a contradiction 
    if $\eta^*_\beta < \eta^*_\alpha$, thus completing the proof of the claim.
  \end{proof}
  
  For each $\gamma < \omega_1$, since $G \cap E^*_{\gamma, 0} \neq \emptyset$, we can find a condition 
  $(c, p) \in G$ that decides the value of $\dot{h}_\gamma$ and $\dot{\bb{A}}_\gamma$, say as 
  $h_\gamma$ and $\bb{A}_\gamma$. Let $\bb{B}^*_\gamma$ be the subalgebra of $\bb{B}$ generated by 
  $\bb{A}_\gamma$ (in $V$).
  
  For each $\alpha \in T^*_1$ and each $n < \omega$, let 
  \[
    b_{\alpha, n} = \bigwedge \{p(\alpha, n) \mid (\alpha, n) \in \dom{p} \text{ and } \exists c \in \bb{C} 
    ~ (c,p) \in G\}.
  \] 
  Note that $b_{\alpha, n}$ is in $\bb{B}^*_\alpha$ (possibly equal to $0$) and, if 
  $b_{\alpha, n} > 0$, then $\mu(b_{\alpha, n}) \geq \frac{1}{2}$.  
  In $V$, let $\dot{H}$ be the canonical $\bb{B}$-name for the generic filter.
  
  \begin{claim}
    For every $\alpha \in T^*_1$, there is $n < \omega$ such that $b_{\alpha, n} > 0$ and 
    $b_{\alpha, n} \in H$.
  \end{claim}
  
  \begin{proof}
    Fix $\alpha \in T^*_1$, and suppose for sake of contradiction that there is $b \in H$ such that, 
    in $V$, we have 
    \begin{itemize}
      \item $b \Vdash_{\bb{B}} ``\alpha \in \dot{T}^*_1"$;
      \item $b \Vdash_{\bb{B}} ``\text{there is no } n < \omega \text{ such that } b_{\alpha, n} > 0 
      \text{ and } b_{\alpha, n} \in \dot{H}"$.
    \end{itemize}
    In particular, for each $n < \omega$ such that $b_{\alpha, n} > 0$, it must be the case 
    that $b_{\alpha, n}$ and $b$ are incompatible. Since each $b_{\alpha, n}$ is in $\bb{B}^*_\alpha$, 
    we can assume that $b$ is in $\bb{B}^*_\alpha$ as well. We can therefore find $k < \omega$ 
    such that $\mu(h_\alpha(k) - b) < \frac{1}{4}\mu(h_\alpha(k))$, and thus 
    $\mu(h_\alpha(k) \wedge b) > \frac{3}{4}\mu(h_\alpha(k))$. Now find $(c, p) \in G \cap 
    E^*_{\alpha, k}$. Then we can find $n < \omega$ such that $(\alpha, n) \in \dom{p}$, 
    $\mu(p(\alpha, n)) > \frac{1}{2}$, and $\mu(p(\alpha, n) - h_\alpha(k)) < 
    \frac{1}{2} - \frac{1}{2} \mu(h_\alpha(k))$. In particular, since $b_{\alpha, n} \leq 
    p(\alpha, n)$ and $\mu(b_{\alpha, n}) \geq \frac{1}{2}$, we must have 
    $\mu(b_{\alpha, n} \wedge h_{\alpha}(k)) > \frac{1}{2}\mu(h_\alpha(k))$. Altogether, this 
    implies that $b_{\alpha, n}$ and $b$ are compatible in $\bb{B}$, contradicting our choice of $b$.
  \end{proof}
  
  For each $\alpha \in T^*_1$, let $f(\alpha)$ be the least $n$ such that $b_{\alpha, n} > 0$ and 
  $b_{\alpha, n} \in H$. For each $\alpha \in T^*_0$, recall that we defined an ordinal 
  $\eta^*_{\alpha}$ such that $\eta^*_\alpha <_T \alpha$ and $\eta^*_\alpha \in T^*_1$. 
  For such $\alpha$, let $f(\alpha) = f(\eta^*_\alpha)$. 
  
  It remains to show that $f$ witnesses that $T$ is $B$-special. We first claim that $f \restriction 
  T^*_1$ witnesses that $T^*_1$ is special. To this end, fix $\alpha, \beta \in T^*_1$ and suppose 
  that $f(\alpha) = f(\beta) = n$. Then we have $b_{\alpha, n} \wedge b_{\beta, n} \in H$. Find 
  $(c,p) \in G$ such that $(\alpha, n)$ and $(\beta, n)$ are both in $\dom{p}$. Then, by the definition 
  of $\bb{P}$ in $V^{\bb{C}}$, we must have 
  \[
    c \Vdash_{\bb{C}} ``p(\alpha, n) \wedge p(\beta, n) \Vdash_{\bb{B}} ``\alpha \not<_{\dot{T}_1} 
    \beta"".
  \]
  Since $\alpha, \beta \in T^*_1$, we can find an extension $c^*$ of $c$ such that 
  $(c^*, p) \in G$ and $(c^*, b_\alpha^* \wedge b_\beta^*) \Vdash ``\alpha, \beta \in \dot{T}_1"$. 
  Then $p(\alpha, n) \wedge p(\beta, n) \wedge b_\alpha^* \wedge b_\beta^*$ is in $H$ and forces 
  that $\alpha \not<_{\dot{T}} \beta$. Therefore, $\alpha$ and $\beta$ are incomparable in $T$, 
  so $f \restriction T^*_1$ really does witness that $T^*_1$ is special. 
  
  Now suppose that $\alpha < \beta < \gamma < \omega_1$, $f(\alpha) = f(\beta) = f(\gamma) = n$, 
  and $\alpha <_T \beta, \gamma$. Suppose first that $\beta \in T^*_1$. Since $f \restriction T^*_1$ 
  witnesses that $T^*_1$ is special, we must have $\alpha \in T^*_0$. But then we have 
  $\eta^*_\alpha <_T \alpha <_T \beta$, $\eta^*_\alpha \in T^*_1$, and $f(\eta^*_\alpha) = f(\alpha) = 
  f(\beta)$, which is again a contradiction. Thus, we must have $\beta \in T^*_0$; similarly, we also have 
  $\gamma \in T^*_0$. 
  
  If $\alpha <_T \eta^*_\beta$, then we reach a contradiction exactly as in the previous paragraph, since 
  $f(\alpha) = f(\beta) = f(\eta^*_\beta)$ and $\eta^*_\beta \in T^*_1$. Therefore, we must have 
  $\eta^*_\beta \leq_T \alpha$. Similarly, $\eta^*_\gamma \leq_T \alpha$. 
  But then $\eta^*_\beta <_T \gamma$ and $\eta^*_\gamma <_T \beta$, so, by Claim \ref{eta_claim}, 
  it must be the case that $\beta <_T \gamma$. Therefore, $f$ witnesses that $T$ is $B$-special.
\end{proof}

As a corollary, we can show that an ``indestructible" version of the negation of the 
weak Kurepa Hypothesis is compatible with any possible value of the continuum except 
$\omega_1$. More precisely:

\begin{corollary} \label{wkh_continuum_cor}
  Suppose that $\PFA$ holds, $\kappa \geq \omega_2$ is a cardinal of uncountable 
  cofinality, and $\bb{B}$ is the measure algebra on $2^\kappa$. Then, in 
  $V^{\bb{B}}$, $2^\omega = \kappa$ and, for every tree $T$ of size and 
  height $\omega_1$ and every outer model $W$ of $V^{\bb{B}}$ such that $(\omega_1)^W = 
  (\omega_1)^{V^{\bb{B}}}$, $T$ has at most $\omega_1$-many uncountable branches in 
  $W$.
\end{corollary}

\begin{proof}
  The fact that $2^\omega = \kappa$ in $V^{\bb{B}}$ is a standard fact about 
  random forcing.
  Let $T \in V^{\bb{B}}$ be a tree of size and height 
  $\omega_1$, and let $W$ be an outer model 
  of $V^{\bb{B}}$ with $(\omega_1)^W = (\omega_1)^{V^{\bb{B}}}$. By Theorem 
  \ref{random_wkh_theorem}, $T$ is $B$-special in $V^{\bb{B}}$. By 
  \cite[Theorem 7.4]{baumgartner_pfa}, it follows that $T$ has at most 
  $\omega_1$-many uncountable branches in $V^{\bb{B}}$. By Proposition 
  \ref{b_special_outer_model_prop}, every uncountable branch of $T$ in $W$ is already 
  in $V$. Therefore, $T$ has at most $\omega_1$-many uncountable branches 
  in $W$.
\end{proof}

We now connect the results in this section back to the main subject of this paper. 
In \cite{cox_krueger_indestructible}, Cox and Krueger introduce the \emph{indestructible 
guessing model property}:

\begin{definition}[\cite{cox_krueger_indestructible}]
  Let $\theta \geq \omega_2$ be a regular cardinal. $M \in \power_{\omega_2} H(\theta)$ 
  is said to be an \emph{indestructible $\omega_1$-guessing model} if it is an 
  $\omega_1$-guessing  model and remains an $\omega_1$-guessing model in any 
  forcing extension that preserves $\omega_1$. $\IGMP(\theta)$ is the assertion 
  that there are stationarily many indestructible $\omega_1$-guessing models in 
  $\power_{\omega_2} H(\theta)$. $\IGMP$ is the 
  assertion that $\IGMP(\theta)$ holds for all regular $\theta \geq \omega_2$.
\end{definition}

In \cite{cox_krueger_indestructible}, 
Cox and Krueger show that $\IGMP$ follows from the conjunction of the following two statements:
\begin{itemize}
  \item for all regular $\theta \geq \omega_2$, there are stationarily many internally unbounded 
  $\omega_1$-guessing models in $\power_{\omega_2} H(\theta)$;
  \item every tree of size and height $\omega_1$ with no cofinal branches is special.
\end{itemize}
In \cite{krueger_sch}, Krueger proves that if $\theta \geq \omega_2$ is a 
regular cardinal and $N \prec H(\theta)$ is an $\omega_1$-guessing model 
with $\omega_1 \subseteq N$, then $N$ is internally unbounded. 
Therefore, $\IGMP$ follows from the conjunction of $\GMP$ and the assertion that all trees of size and 
height $\omega_1$ with no cofinal branches are special; in particular, it follows from $\PFA$.

$\IGMP$ has some consequences that $\GMP$ does not have. For example, it implies that there are no 
Suslin trees, whereas other work of Cox and Krueger \cite{cox_krueger_quotients} shows that 
$\GMP$ is compatible with the existence of a Suslin tree (we give another model for this 
conjunction in the next section by showing that $\GMP$ follows from the forcing 
axiom $\PFA(S)$).

By the argument of \cite[Theorem 2.8]{cox_krueger_indestructible}, it is clear that 
$\IGMP(\omega_2)$ implies the indestructible version of $\neg \wKH$ isolated in 
Corollary \ref{wkh_continuum_cor} in the case in which $W$ is a forcing extension of 
$V$. In Corollary \ref{wkh_continuum_cor}, we saw that this indestructible version of 
$\neg \wKH$ is compatible with any possible value of the continuum, including 
values of cofinality $\omega_1$. In \cite{cox_krueger_indestructible}, Cox and 
Krueger show that $\IGMP$ is compatible with any possible value of the continuum 
with cofinality at least $\omega_2$. The combination of these two results naturally 
raises the following question:

\begin{question}
  Is $\IGMP$ compatible with $\cf(2^\omega) = \omega_1$? What about just 
  $\IGMP(\omega_2)$?
\end{question}

\section{Forcing axioms for Suslin and almost Suslin trees} \label{pfas_section}

In this section, we continue investigations motivated by the study of $\IGMP$, 
connecting it with the forcing axioms $\PFA(S)$ and $\PFA(T^*)$, introduced 
by Todorcevic \cite{todorcevic_forcing} and Krueger \cite{krueger_forcing_axiom}, 
respectively. In the process, we answer
questions of Cox and Krueger \cite{cox_krueger_indestructible} and Krueger 
\cite{krueger_forcing_axiom}. Let us begin by recalling the relevant definitions.

\begin{definition}
  Suppose that $T$ is an $\omega_1$-tree, i.e., a tree of height $\omega_1$, all of 
  whose levels are countable.
  \begin{enumerate}
    \item $T$ is an \emph{Aronszajn tree} if it has no cofinal branches.
    \item $T$ is a \emph{Suslin tree} if it is an Aronszajn tree and has no 
    uncountable antichains.
    \item $T$ is an \emph{almost Suslin tree} if it has no stationary antichains, 
    i.e., no antichains $A \subseteq T$ for which the set $\{\height(s) \mid s 
    \in A\}$ is stationary in $\omega_1$, where $\height(s)$ denotes the level of 
    $s$ in $T$.
  \end{enumerate}
\end{definition}

We note that an almost Suslin tree need not be an Aronszajn tree, though in this 
paper we will only be interested in almost Suslin trees that are Aronszajn. 
Throughout this section, in accordance with established notation, we will always 
use $S$ to denote a Suslin tree and $T^*$ to denote an almost Suslin Aronszajn tree. 
With this convention, there will be no confusion introduced by the following slight 
abuse of notation.

\begin{definition}
  Let $\P$ be a forcing notion.
  \begin{enumerate}
    \item For a Suslin tree $S$, we say that $\P$ is \emph{$S$-preserving} if 
    $\Vdash_{\P}``S \text{ is a Suslin tree}"$.
    \item For an almost Suslin Aronszajn tree $T^*$, we say that $\P$ is 
    \emph{$T^*$-preserving} if $\Vdash_{\P}``T^* \text{ is an almost Suslin 
    Aronszajn tree}"$.
  \end{enumerate}
\end{definition}

We now introduce the forcing axioms that form the subject of this section. We note 
that in most other works, the forcing axioms $\MA_{\omega_1}(S)$ and 
$\PFA(S)$ require that $S$ be a \emph{coherent} Suslin tree. Since coherence 
will not be necessary in any of our arguments, we choose to present the axioms 
in a more general form and not require coherence. We note that, when
we interpret a Suslin tree as a forcing notion, then the forcing order is
understood to be the reverse of the tree order.

\begin{definition}
  If $\mc C$ is a class of forcing posets, then $\mathsf{FA}(\mc C)$ is the assertion 
  that, for every $\P \in \mc C$ and every collection $\mc D = \{D_\alpha \mid \alpha < 
  \omega_1\}$ of $\omega_1$-many dense subsets of $\P$, there is a filter $G \subseteq \P$ 
  such that $G \cap D_\alpha \neq \emptyset$ for all $\alpha < \omega_1$.
  \begin{enumerate}
    \item $\MA_{\omega_1}(S)$ is the assertion that $S$ is a Suslin 
    tree and $\mathsf{FA}(\mc C)$ 
    holds, where $\mc C$ is the class of c.c.c.\ $S$-preserving posets.
    \item $\PFA(S)$ is the assertion that $S$ is a Suslin tree and 
    $\mathsf{FA}(\mc C)$ holds, where $\mc C$ is the class of proper $S$-preserving posets.
    \item If we start with a model satisfying $\PFA(S)$ and then force with the
    Suslin tree $S$, then we say that the resulting forcing extension satisfies 
    $\PFA(S)[S]$. Asserting that $\PFA(S)[S]$ implies a statement $\varphi$ should be 
    understood as asserting that, in any model of $\ZFC$ satisfying $\PFA(S)$ for 
    some coherent Suslin tree $S$, we have $\Vdash_{S} \varphi$. $\MA_{\omega_1}(S)[S]$ 
    is defined analogously, with $\MA_{\omega_1}(S)$ replacing $\PFA(S)$.
    \item $\PFA(T^*)$ is the assertion that $T^*$ is an almost Suslin Aronszajn 
    tree and $\mathsf{FA}(\mc C)$ holds, where $\mc C$ is the class of 
    proper $T^*$-preserving posets.
  \end{enumerate}
\end{definition}

\begin{remark}
  The consistency of $\PFA(S)$ (and hence of $\PFA(S)[S]$) and of 
  $\PFA(T^*)$ follows from the consistency 
  of the existence of a supercompact cardinal in much the same way that the 
  consistency of $\PFA$ follows from the same 
  hypothesis. For sketches of the proof, see \cite[Theorem 4.1]{todorcevic_forcing}
  and \cite[Theorem 2.6]{krueger_forcing_axiom}.
\end{remark}

The following lemma, due to Woodin (cf.\ \cite[Proof of Theorem 2.53]{woodin}) 
will be useful. Recall that, if $\P$ is a forcing poset and $\P \in M \prec H(\theta)$ 
for some regular uncountable cardinal $\theta$, then we say that a filter 
$G \subseteq \P$ is \emph{$M$-generic} if $G \cap D \cap M \neq \emptyset$ for 
all dense subsets $D$ of $\P$ that are elements of $M$.

\begin{lemma} \label{woodin_lemma}
  Suppose that $\mc C$ is a class of forcing posets and $\mathsf{FA}(\mc C)$ holds. 
  Then, for every poset $\P \in \mc C$ and all sufficiently large regular 
  cardinals $\theta$, the set
  \[
    \{M \in \power_{\omega_2} H(\theta) \mid |M| = \omega_1 \subseteq M \wedge 
    \exists G \subseteq \P ~ [G \text{ is an } M\text{-generic filter}]\}
  \]
  is stationary in $\power_{\omega_2} H(\theta)$.
\end{lemma}

\subsection{Suslin trees} \label{suslin_subsec}
In \cite{cox_krueger_indestructible}, Cox and Krueger ask whether $\IGMP$ implies $\mathfrak{p} > \omega_1$. 
In this section, we prove that the axiom $\PFA(S)[S]$ implies $\IGMP$. 
Since $\mathfrak{p} = \omega_1$ in any forcing extension by a Suslin tree 
(cf.\ \cite[Lemma 2]{farah}), and hence in any model of $\PFA(S)[S]$, 
this answers Cox and Krueger's question negatively. We also show that 
$\PFA(S)$ implies $\GMP$. Thus, models of $\PFA(S)$ provide examples of models in which $\GMP$ holds and 
there exists a Suslin tree, and therefore $\IGMP$ fails, but forcing with a particular 
Suslin tree yields a model of $\IGMP$.

%
%
%

Given a tree $T$ of height $\omega_1$ and a set $C \subseteq \omega_1$, 
let $T_C$ denote the set of $s \in T$ such that $\height(s) \in C$, with the 
inherited ordering from $T$. Note that, if $S$ is a Suslin tree and $C\subseteq 
\omega_1$ is unbounded, then $S_C$ is itself a Suslin tree.

The following theorem is due to Raghavan and Yorioka \cite{raghavan_yorioka} 
under the additional assumption that $S$ is coherent,
though they indicate that it was known beforehand in the case in which $\dot{T}$ is a name for
an Aronszajn tree. Since we remove the requirement that $S$ be coherent, 
and because we find our proof to be simpler
than that presented in \cite{raghavan_yorioka}, we include a proof.

\begin{theorem} \label{ma_s_special_tree_thm}
  Suppose that $S$ is a Suslin tree and $\dot{T}$ is an $S$-name for a tree of
  height and size $\omega_1$ with no cofinal branches. Then there is a forcing notion
  $\bb{P}(\dot{T})$ such that $\bb{P}(\dot{T}) \times S$ is c.c.c.\ and
  $\Vdash_{\bb{P}(\dot{T}) \times S} ``\dot{T} \text{ is special"}$.
\end{theorem}

\begin{proof}
  We may assume that it is forced by $S$ that:
  \begin{itemize}
    \item the underlying set of $\dot{T}$ is $\omega_1$;
    \item for all $\alpha, \beta < \omega_1$, if $\alpha <_{\dot{T}} \beta$, then
    $\alpha < \beta$.
  \end{itemize}
  We can therefore think of $\dot{T}$ as a name for a subset $\dot{X} \subseteq
  [\omega_1]^2$, where, for all $\alpha < \beta < \omega_1$, we have
  \[
    \Vdash_S ``\alpha <_{\dot{T}} \beta \Leftrightarrow \{\alpha, \beta\} \in \dot{X}".
  \]
  By the argument immediately after the statement of Theorem 4.2 of
  \cite{larson_todorcevic}, we can find an unbounded $C \subseteq \omega_1$ and
  a subset $K \subseteq [S_C]^2$ such that for all $t \in S_C$ and all
  $\alpha < \beta < \height_{S_C}(t)$, we have $t \Vdash_S ``\{\alpha, \beta\} \in \dot{X}"$
  if and only if $\{s,r\} \in K$, where $s$ and $r$ are the predecessors of $t$
  at levels $\alpha$ and $\beta$ of $S_C$, respectively. Since forcing with $S_C$ is
  equivalent to forcing with $S$, for ease of notation we will assume that $C = \omega_1$
  and write $S$ instead of $S_C$. Also, since the preceding sentences only mention
  pairs $\{s,r\} \in K$ such that $s <_S r$ or $r <_S s$, we can assume that $K$ only
  consists of pairs that are comparable via $<_S$.

  Let $\bb{P}(\dot{T}) = \bb{P}$ be the forcing notion whose conditions are all
  finite partial functions
  $p:S \rightarrow \omega$ with the property that, for all $s,t \in \dom{p}$, if
  $p(s) = p(t)$, then $\{s,t\} \notin K$. Order $\bb{P}$ by reverse inclusion. We can
  think of $\bb{P}$ as the natural forcing to add an $S$-name for a specializing function
  for $\dot{T}$.

  \begin{claim}
    $\bb{P} \times S$ is c.c.c.
  \end{claim}

  \begin{proof}
    Let $\langle (p_\eta, s_\eta) \mid \eta < \omega_1 \rangle$ be a sequence of conditions
    in $\bb{P} \times S$. We will prove that there are $\eta_0 < \eta_1$ such that
    $(p_{\eta_0}, s_{\eta_0})$ and $(p_{\eta_1}, s_{\eta_1})$ are compatible.
    Since $S$ is Suslin, and hence c.c.c., there is
    $s^* \in S$ such that $s^* \Vdash_S ``\{\eta < \omega_1 \mid s_\eta \in \dot{G}\}
    \text{ is uncountable},"$ where $\dot{G}$ is the canonical name for the
    $S$-generic filter. Let $G \subseteq S$ be a $V$-generic filter with $s^* \in G$,
    and move to $V[G]$. Let $A := \{\eta < \omega_1 \mid s_\eta \in G\}$. By our choice
    of $G$, $A$ is uncountable. Using the $\Delta$-system lemma, find an unbounded
    $A' \subseteq A$ such that $\{\dom{p_\eta} \mid \eta \in A'\}$ forms a $\Delta$-system,
    with root $R$. Since each level of $S$ is countable and the codomain of each $p_\eta$
    is $\omega$, by thinning out $A'$ further if necessary, we may also assume the
    following.
    \begin{itemize}
      \item There is $n < \omega$ such that, for all $\eta \in A'$,
      $R_\eta := \dom{p_\eta} \setminus R$ has size $n$. Enumerate each $R_\eta$ as
      $\{r^\eta_m \mid m < n\}$.
      \item For all $\eta_0, \eta_1 \in A'$, we have $p_{\eta_0} \restriction R
      = p_{\eta_1} \restriction R$.
      \item For all $\eta_0 < \eta_1$, both in $A'$, for all $m_0, m_1 < n$, and for
      all $r \in R$, we have $\height(r) < \height(r^{\eta_0}_{m_0}) <
      \height(r^{\eta_1}_{m_1})$.
    \end{itemize}
    For all $\eta_0 < \eta_1$ in $A'$, we know that $s_{\eta_0}$ and $s_{\eta_1}$ are
    both in $G$ and are therefore compatible. Therefore, if we can find such $\eta_0$
    and $\eta_1$ for which $p_{\eta_0}$ and $p_{\eta_1}$ are compatible, then we will
    be finished. Suppose that this is not the case. We proceed as in Baumgartner's
    classical proof that the forcing to specialize an $\omega_1$-tree with no cofinal
    branch is c.c.c. Namely, for each pair $\eta < \xi$, both from $A'$, our uniformization
    of $A'$ implies that the only obstacle to the compatibility of $p_\eta$ and $p_\xi$
    would be the existence of $m_{\eta,\xi,0}, m_{\eta,\xi,1} < n$ such that
    $p_\eta(r^\eta_{m_{\eta,\xi,0}}) = p_\xi(r^\xi_{m_{\eta,\xi,1}})$ and
    $\{r^\eta_{m_{\eta,\xi,0}}, r^\xi_{m_{\eta,\xi,1}}\} \in K$. It follows from our
    assumptions on $K$ that $r^\eta_{m_{\eta,\xi,0}} <_S r^\xi_{m_{\eta,\xi,1}}$.

    Let $U$ be a uniform ultrafilter over $A'$. For each $\eta \in A'$, we can fix
    a set $Y_\eta \in U$ and numbers $m_{\eta,0}, m_{\eta,1} < n$ such that, for all
    $\xi \in Y_\eta$, we have $\eta < \xi$, $m_{\eta,\xi,0} = m_{\eta,0}$, and
    $m_{\eta,\xi,1} = m_{\eta,1}$. Now find an unbounded $A'' \subseteq A'$ and
    numbers $m_0, m_1 < n$ such that $m_{\eta,0} = m_0$ and $m_{\eta,1} = m_1$ for all
    $\eta \in A''$. Given $\eta_0 < \eta_1$, both in $A''$, find $\xi \in Y_{\eta_0}
    \cap Y_{\eta_1}$. Then $r^{\eta_i}_{m_0} <_S r^{\xi}_{m_1}$ for each $i < 2$,
    so we have $r^{\eta_0}_{m_0} <_S r^{\eta_1}_{m_0}$. Also, $\{r^{\eta_i}_{m_0},
    r^{\xi}_{m_1}\} \in K$ for each $i < 2$, so, by our assumptions on $\dot{T}$,
    we must have $\{r^{\eta_0}_{m_0}, r^{\eta_1}_{m_0}\} \in K$, as well.

    By the previous paragraph, the downward closure of $\{r^\eta_{m_0} \mid \eta \in
    A''\}$ is a cofinal branch through $S$; let $b'$ denote this cofinal branch.
    Since $S$ is Suslin in $V$, $b'$ generates a $V$-generic filter $G'$ over $S$. Let
    $T'$ be the interpretation of $\dot{T}$ using $G'$. For all $\eta_0 < \eta_1$
    in $A''$, we have $\{r^{\eta_0}_{m_0}, r^{\eta_1}_{m_0}\} \in K$, so the
    $<_{T'}$-downward closure of $\{\height(r^\eta_{m_0}) \mid \eta \in A''\}$
    generates a cofinal branch in $T'$, contradicting the fact that, in $V$, we
    have $\Vdash_S ``\dot{T} \text{ has no cofinal branches}"$.
  \end{proof}

  For each $s \in S$, let $D_s := \{p \in \bb{P} \mid s \in \dom{p}\}$, and note that
  $D_s$ is a dense open subset of $\bb{P}$. Therefore, the generic object for
  $\bb{P}$ can be seen as a function from $S$ to $\omega$. Let $\dot{g}$ be a
  canonical $\bb{P}$-name for this function, and let $\dot{c}$ be a $\bb{P} \times
  S$-name for a function from $\omega_1$ to $\omega$ such that, for all $\alpha < \omega$
  and all $t \in S$ with $\height(t) > \alpha$, we have
  \[
    (\emptyset, t) \Vdash_{\bb{P} \times S} ``\dot{c}(\alpha) = \dot{g}(s)",
  \]
  where $s$ is the predecessor of $t$ on level $\alpha$ of $S$.

  We will be done if we show that $\dot{c}$ is forced to be a specializing function for
  $\dot{T}$. Suppose for sake of contradiction that there are $(p,t) \in \bb{P} \times S$
  and $\alpha < \beta < \omega_1$ such that $(p,t) \Vdash_{\bb{P} \times S}
  ``\alpha <_{\dot{T}} \beta
  \text{ and } \dot{c}(\alpha) = \dot{c}(\beta)"$. We can assume that $\height(t) >
  \beta$. Let $s$ and $r$ be the predecessors of $t$ on levels $\alpha$ and $\beta$ of $S$,
  respectively. We can also assume that $\{s,r\} \subseteq \dom{p}$.
  Then we have $\{s,r\} \in K$ and $p(s) = p(r)$,
  contradicting the fact that $p$ is a condition in $\bb{P}$.
\end{proof}

\begin{corollary} \label{ma_s_special_tree_cor}
  $\MA_{\omega_1}(S)[S]$ implies that every tree of height and size $\omega_1$ with no cofinal
  branches is special.
\end{corollary}

We now prove that $\PFA(S)[S]$ implies $\GMP$. Since $\PFA(S)[S]$ also
implies that all trees of height and size $\omega_1$ with no cofinal branches 
are special, it will follow from the discussion at the end of Section~\ref{special_sec} 
that $\PFA(S)[S]$ implies $\IGMP$. Our proof is a modification of the proof
from \cite{viale_weiss} of the fact that $\PFA$ implies $\GMP$. 

We first recall the \emph{covering} and \emph{approximation} properties, introduced by 
Hamkins (cf.\ \cite{hamkins_approximation}).

\begin{definition}
  Suppose that $V \subseteq W$ are transitive models of $\ZFC$ and $\mu$ is a regular uncountable cardinal.
  \begin{enumerate}
    \item $(V,W)$ satisfies the \emph{$\mu$-covering property} if, for every $x \in W$ such that 
    $x \subseteq V$ and $|x|^W < \mu$, there is $y \in V$ such that $|y|^V < \mu$ and $x \subseteq y$.
    \item $(V,W)$ satisfies the \emph{$\mu$-approximation property} if, for all $x \in W$ such that 
    $x \subseteq V$ and $x \cap z \in V$ for all $z \in V$ with $|z| < \mu$, we in fact have $x \in V$.
  \end{enumerate}
  A poset $\P$ has the \emph{$\mu$-covering property} (resp.\ \emph{$\mu$-approximation 
  property}) if, for every $V$-generic filter $G \subseteq \P$, the pair $(V,V[G])$ has the 
  $\mu$-covering property (resp.\ $\mu$-approximation property).
\end{definition}

For the rest of this subsection, let $S$ denote a Suslin tree. 
If $\dot{T}$ is an $S$-name for a tree of height and size $\omega_1$ with no
cofinal branch, then $\bb{P}(\dot{T})$ denotes the forcing from Theorem
\ref{ma_s_special_tree_thm} that adds an $S$-name for a specializing function for
$\dot{T}$. If $\P$ is an $S$-preserving forcing notion and $\dot{T}$ 
is a $\P \times S$-name for 
a tree of height and size $\omega_1$ with no cofinal branch, then we will 
typically let $\dot{\P}(\dot{T})$ be a $\P$-name for $\P(\dot{T})$ (where, in 
$V^{\P}$, $\dot{T}$ is reinterpreted as an $S$-name). 
The following is analogous to, and largely follows the
proof of, \cite[Lemma 4.6]{viale_weiss}.

\begin{lemma} \label{technical_lemma}
  Suppose that $\lambda \geq 2^\omega$ is an infinite cardinal, and let
  $\theta$ be a sufficiently large
  regular cardinal. Assume that $\bb P$ preserves $S$, collapses $2^\lambda$
  to have cardinality
  $\omega_1$, satisfies the $\omega_1$-covering and $\omega_1$-approximation properties,
  and continues to satisfy the $\omega_1$-approximation property in $V^S$.
  Then there is a $\P \times S$-name $\dot{T}_1$ for a tree of height
  and size $\omega_1$ with no cofinal branches and a $w \in H(\theta)$ such that, for every
  $M \in \power_{\omega_2} H(\theta)$ such that $\omega_1 \cup \{w\} \subseteq M \prec
  H(\theta)$, if there is $G \subseteq \bb{P} \ast \dot{\bb{P}}(\dot{T}_1)$ that is
  $M$-generic, then, in $V^S$, $M^S$ is an $\omega_1$-guessing model for
  $\lambda$.
\end{lemma}

\begin{proof}
  Work for now in $V$, and
  let $\dot{B}$ be an $S$-name for $({^{\lambda}}2)^{V^{S}}$. Using the fact that 
  $\P$ collapses $2^\lambda$ to have cardinality $\omega_1$ and satisfies the 
  $\omega_1$-covering property, let $\dot{c}$ be 
  a $\P$-name for a $\subseteq$-increasing, continuous, and cofinal function from 
  $\omega_1$ to $(\power_{\omega_1} \lambda)^{V^{\P}}$. Since $\P$ satisfies the 
  $\omega_1$-covering property, we can assume that, for all $\alpha < \omega_1$, 
  we have $\Vdash_{\P}``\dot{c}(\alpha+1) \in V"$. Let $\dot{\ell}$ be a 
  $\P \times S$-name for a bijection from $\omega_1$ to $B$, and 
  let $\dot{T}$ be a $\P \times S$-name 
  for $\{\dot{\ell}(\eta) \restriction \dot{c}(\alpha) \mid \eta, \alpha < \omega_1\}$. 
  Note that $\dot{T}$ is forced to be a tree of height and size $\omega_1$. 
  Moreover, since $\P$ has the $\omega_1$-approximation property in $V^S$, $\dot{B}$ 
  is forced to be precisely the set of cofinal branches through $\dot{T}$. 
  Since $\dot{B}$ is forced to have cardinality $\omega_1$ in $V^{\P \times S}$, we 
  can apply Lemma \ref{baumgartner_function_lemma} to find a $\P \times S$-name 
  $\dot{g}$ for a Baumgartner function from $\dot{B}$ to $\dot{T}$.

  Let $\dot{T}_0$ be a $\P \times S$-name for the set $\{t \in \dot{T} \mid 
  \exists b \in \dot{B} ~ \dot{g}(b) <_{\dot{T}} t \in b\}$, and let 
  $\dot{T}_1$ be a $\P \times S$-name for $\dot{T} \setminus \dot{T}_0$.
  Since $\dot{T}_0$ is forced to contain a tail of every $b \in \dot{B}$, it follows 
  that $\dot{T}_1$ is forced to have no uncountable branches in $V^{\P \times S}$. 
  We can therefore let $\dot{\Q}$ be a $\P$-name for $\P(\dot{T}_1)$, where 
  $\dot{T}_1$ is reinterpreted as an $S$-name in $V^{\P}$. Since $\dot{\Q}$ is 
  forced to add an $S$-name for a specializing function for $\dot{T}_1$, we can 
  fix a $(\P \ast \dot{\Q}) \times S$-name $\dot{f}$ for a specializing function 
  from $\dot{T}_1$ to $\omega$.

  We claim that $\dot{T}_1$ is as desired. 
  Let $w$ be a set containing all relevant information, including
  $S, \dot{c}$, $\dot{T}$, $\dot{T}_0$, $\dot{T}_1$, $\dot{f}$, $\dot{g}$, 
  and $\dot{\ell}$. Now fix $M \in \power_{\omega_2} H(\theta)$
  such that $\omega_1 \cup \{w\} \subseteq M \prec H(\theta)$, and suppose that
  $G = G_0 \ast G_1 \subseteq (\bb{P} \ast \dot{\Q}) \cap M$ is $M$-generic.
  We will show that, in $V^{S}$, $M^S$ is an $\omega_1$-guessing model for $\lambda$.

  Let $H$ be an $S$-generic filter over $V$ and move to $V[H]$. Let
  $c := \dot{c}^{G_0}: \omega_1 \rightarrow \power_{\omega_1}(\lambda \cap M)$,
  and note that $c$ is $\subseteq$-increasing and continuous, and $c(\alpha + 1) \in
  M$ for all $\alpha < \omega_1$. Also, $c$ is cofinal in $(\power_{\omega_1}\lambda)^M$,
  and, since $S$ is $\omega_1$-distributive, it is also cofinal in
  $(\power_{\omega_1}\lambda)^{M[H]}$. Let $g := \dot{g}^{G_0 \times H}$,
  $T := \dot{T}^{G_0 \times H}$ (and analogously
  for $T_0$ and $T_1$), $f := \dot{f}^{(G_0 \ast G_1) \times H}$, $B = \dot{B}^H$, and
  $\ell := \dot{\ell}^{G_0 \times H}$. Let $B \restriction M[H] := \{b \restriction
  M[H] \mid b \in B \cap M[H]\}$.

  By elementarity and the fact that $G_0 \ast G_1$ is $M$-generic, we have the following
  facts:
  \begin{itemize}
    \item $\ell : \omega_1 \rightarrow B \cap M[H]$ is a bijection.
    \item $T = \{b \restriction c(\alpha) \mid b \in B \cap M[H], ~ \alpha < \omega_1\}$.
    \item $B \restriction M[H]$ is a set of uncountable branches through $T$.
    \item $\dom{g} = B \cap M[H]$; we will slightly abuse notation and identify
    $\dom{g}$ with $B \restriction M[H]$.
    \item $g : B \restriction M[H] \rightarrow T$ is a Baumgartner function.
    \item $T = T_0 \cup T_1$.
    \item $f:T_1 \rightarrow \omega$ is a specializing function.
  \end{itemize}

  \begin{claim} \label{unctble_branch_claim}
    $B \restriction M[H]$ is the set of uncountable branches through $T$.
  \end{claim}

  \begin{proof}
    We have already seen that every element of $B \restriction M[H]$ is an uncountable branch
    through $T$. For the reverse inclusion, suppose that $h$ is an uncountable branch
    through $T$. We identify branches through $T$ with their union, i.e., we
    think of $h$ as being of the form $h:M[H] \cap \lambda \rightarrow 2$.
    Since $T_1$ is special and therefore cannot have an uncountable
    branch, a tail of $h$ must lie inside $T_0$, i.e., we can find $\alpha_0 < \omega_1$
    such that, for all $\alpha \in [\alpha_0, \omega_1)$, we have $h \restriction c(\alpha)
    \in T_0$. For all such $\alpha$, there is a unique $b_\alpha \in B \restriction M[H]$
    such that $g(b_\alpha) \subsetneq h \restriction c(\alpha) \subsetneq b_\alpha$.
    By Fodor's Lemma, there is a fixed $t \in T$ and a stationary $R \subseteq \omega_1$
    such that $g(b_\alpha) = t$ for all $\alpha \in R$. Since $g$ is injective, there is
    $b \in B \restriction M[H]$ such that $b_\alpha = b$ for all $\alpha \in R$. Then
    $h \restriction c(\alpha) \subseteq b$ for all $\alpha \in R$, and since
    $R$ is cofinal in $\omega_1$, this implies that $h = b$.
  \end{proof}

  We are now ready to prove that $M[H]$ is an $\omega_1$-guessing model for $\lambda$
  in $V[H]$. Work in $V[H]$, and let $d \subseteq \lambda$ be $M[H]$-approximated.
  Let $h:\lambda \cap M[H] \rightarrow 2$ be the characteristic function of
  $d \cap M[H]$. Then, for all $\alpha < \omega_1$, we have $d \cap c(\alpha+1) \in M[H]$
  and hence $h \restriction c(\alpha+1) \in M[H]$. It follows that $h \restriction
  c(\alpha+1) \in T$, and hence $h$ is an uncountable branch through $T$. By
  Claim \ref{unctble_branch_claim}, we have $h \in B \restriction M[H]$, so there is
  $b \in B \cap M[H]$ such that $b \restriction M[H] = h$. Let
  $e := \{\eta < \lambda \mid b(\eta) = 1\}$. Then $e \in M[H]$ and $e \cap
  M[H] = d \cap M[H]$. Therefore, $d$ is guessed by $M[H]$, as desired.
\end{proof}

\begin{theorem} \label{pfa_s_s_isp_thm}
  $\PFA(S)[S]$ implies $\GMP$.
\end{theorem}

\begin{proof}
  Suppose that $V$ satisfies $\PFA(S)$. We will prove that $\GMP$ holds
  in $V^S$. Given a cardinal $\lambda \geq \omega_2$, let $\GMP^\lambda$ 
  denote
  the assertion that, for all sufficiently large regular $\theta$, there are
  stationarily many $M \in \power_{\omega_2} H(\theta)$ that are $\omega_1$-guessing models
  for $\lambda$. By \cite[Proposition 3.2]{viale_weiss}, for every sufficiently large
  regular $\theta$, if $\GMP^{|H(\theta)|}$ holds, then there are
  stationarily many $\omega_1$-guessing models $M \in \power_{\omega_2} H(\theta)$.
  It therefore suffices to prove that $\GMP^\lambda$ holds for all
  $\lambda \geq 2^\omega$.

  Work in $V$.
  Fix $\lambda \geq 2^\omega$ and a sufficiently large regular cardinal 
  $\theta$, and let $\bb{P} := \mathrm{Add}(\omega, 1) \ast
  \dot{\mathrm{Coll}}(\omega_1, 2^\lambda)$. Then $\bb{P}$ collapses $2^\lambda$
  to have cardinality $\omega_1$ and, since $\mathrm{Add}(\omega, 1)$ is
  $\omega_1$-Knaster and $\dot{\mathrm{Coll}}(\omega_1, 2^\lambda)$ is forced to
  be countably closed, $\bb{P}$ is proper and preserves the fact that $S$ is a Suslin tree.
  It is proven in \cite{krueger_mitchell_iteration}
  that $\bb{P}$ has the $\omega_1$-covering and $\omega_1$-approximation properties.
  Since $S$ is $\omega_1$-distributive, the definition of $\bb{P}$ is the same in
  $V$ and in $V^S$; in particular, the proof from \cite{krueger_mitchell_iteration}
  can be carried out in $V^S$ to show that $\bb{P}$ has the $\omega_1$-approximation
  property in that model as well. We can therefore apply Lemma \ref{technical_lemma}
  to find a $\bb{P} \times S$-name $\dot{T}$ for a tree of height and size $\omega_1$
  with no cofinal branches and a $w \in H(\theta)$ such that, for every
  $M \in \power_{\omega_2} H(\theta)$ such that $\omega_1 \cup \{w\} \subseteq M
  \prec H(\theta)$, if there is $G \subseteq \bb{P} \ast \dot{\bb{P}}(\dot{T})$
  that is $M$-generic, then $M^S$ is an $\omega_1$-guessing model for $\lambda$
  in $V^S$.

  In $V^\bb{P}$, by Theorem \ref{ma_s_special_tree_thm}, $\P(\dot{T}) \times
  S$ is c.c.c. In particular, it follows that, in $V$, $\bb{P} * \dot{\bb{P}}(\dot{T})$
  is proper and preserves $S$. Therefore, by $\PFA(S)$ and Lemma~\ref{woodin_lemma}, 
  the set of $M \in \power_{\omega_2} H(\theta)$ for which $|M|= \omega_1 \subseteq M$ 
  and there exists an $M$-generic $G \subseteq \bb{P} * \dot{\bb{P}}(\dot{T})$
  is stationary in $\power_{\omega_2} H(\theta)$. Let $\dot{C}$ be an $S$-name for a
  club in $(\power_{\omega_2} H(\theta))^{V^S}$, and, using Proposition~\ref{menas_prop}, 
  let $\dot{f}$ be an $S$-name for
  a function from $[H(\theta)]^2 \rightarrow \power_{\omega_2} H(\theta)$
  such that $\dot{C}_{\dot{f}} := \{X \in \power_{\omega_2} H(\theta) \mid
  \forall z \in \power_\omega X ~ \dot{f}(z) \subseteq X\}$ is forced to be
  a subset of $\dot{C}$. We can then find $M \in \power_{\omega_2}H(\theta)$
  such that $\omega_1 \cup \{w, \dot{f}\} \subseteq M \prec H(\theta)$ and such that
  there exists an $M$-generic $G \subseteq \bb{P} * \dot{\bb{P}}(\dot{T})$.

  It follows that $M^S$ is an $\omega_1$-guessing model for $\lambda$ in $V^S$.
  We now show that $M^S \in \dot{C}$; since $\dot{C}$ was an arbitrary name for a club
  in $\power_{\omega_2} H(\theta)$, this will imply that $\GMP^\lambda$ holds
  in $V^S$. To see that $M^S \in \dot{C}$, recall that $\dot{f} \in M$, and therefore,
  for every $S$-name $\dot{z} \in M$ for an element of $\power_\omega H(\theta)$,
  there is an $S$-name for $\dot{f}(\dot{z})$ in $M$. It follows that
  $M^S \in \dot{C}_{\dot{f}} \subseteq \dot{C}$.
\end{proof}

\begin{corollary}
  $\PFA(S)[S]$ implies $\IGMP$.
\end{corollary}


We end this subsection by pulling back the previous results from $V^S$ to $V$, showing 
that $\PFA(S)$ also implies $\GMP$.

\begin{proposition} \label{guessing_pullback_prop}
  Suppose that $\P$ is a forcing notion such that $|\P| \leq \omega_1$ and $\P$ has the 
  $\omega_1$-covering property, and suppose that $\theta$ is a sufficiently large regular 
  cardinal and $M \in \power_{\omega_2} H(\theta)$ is such that $M \prec H(\theta)$ and 
  $\P \cup \{\P\} \subseteq M$. If $M^{\P}$ is an $\omega_1$-guessing model in $V^{\P}$, 
  then $M$ is an $\omega_1$-guessing model in $V$.
\end{proposition}

\begin{proof}
  Suppose that $M^{\P}$ is an $\omega_1$-guessing model in $V^{\P}$ but $M$ is not an 
  $\omega_1$-guessing model in $V$. We can then fix a cardinal $\lambda \in M$ and a 
  set $d \subseteq \lambda$ such that $d$ is $(\omega_1, M)$-approximated but not 
  $M$-guessed.
  
  \begin{claim}
    In $V^{\P}$, $d$ is $(\omega_1, M^{\P})$-approximated.
  \end{claim}
  
  \begin{proof}
    Fix a condition $p \in \P$ and a $\P$-name $\dot{y} \in M$ for an element of 
    $\power_{\omega_1} \lambda$. Since $\P$ has the $\omega_1$-covering property, 
    and by elementarity, we can find $q \leq_{\P} p$ and $z \in \power_{\omega_1} \lambda 
    \cap M$ such that $q \Vdash_{\P}``\dot{y} \subseteq z"$. Then 
    $d \cap z \in M$, so $q \Vdash_{\P}``d \cap \dot{y} = (d \cap z) \cap \dot{y} \in 
    M^{\P}"$. The conclusion follows by genericity.
  \end{proof}
  
  Therefore, since $M^{\P}$ is an $\omega_1$-guessing model in $V^{\P}$, we can find 
  a $\P$-name $\dot{e} \in M$ for a subset of $\lambda$ and a condition $p \in \P$ such 
  that $p \Vdash_{\P}``d \cap M^{\P} = \dot{e} \cap M^{\P}"$. Note that, since $\P \subseteq 
  M$, we have $M^{\P} \cap \lambda = M \cap \lambda$.
  
  \begin{claim}
    There are conditions $q_0, q_1 \leq_{\P} p$ and an ordinal $\alpha < \lambda$ 
    such that $q_0 \Vdash_{\P}``\alpha \in \dot{e}"$ and $q_1 \Vdash_{\P}``\alpha \notin 
    \dot{e}"$.
  \end{claim}
  
  \begin{proof}
    If not, then, letting $e^* := \{\alpha < \lambda \mid \exists q \leq_{\P} p ~ 
    [q \Vdash_{\P}``\alpha \in \dot{e}]\}$, we have that $e^* \in M$ and 
    $p \Vdash_{\P} ``\dot{e} = e^*"$. But then we would have $e^* \cap M = d \cap M$, 
    contradicting the assumption that $d$ is not $M$-guessed.
  \end{proof}
  
  By elementarity, we can find $q_0, q_1 \leq_{\P} p$ and an ordinal $\alpha \in \lambda \cap 
  M$ such that $q_0 \Vdash_{\P}``\alpha \in \dot{e}"$ and $q_1 \Vdash_{\P}``\alpha \notin 
  \dot{e}"$. If $\alpha \in d$, then let $q^* = q_1$, and if $\alpha \notin d$, then 
  let $q^* = q_0$. In either case, we have $q^* \leq_{\P} p$ and 
  $q^* \Vdash_{\P}``\dot{e} \cap M \neq d \cap M"$, which is a contradiction.
\end{proof}

\begin{corollary}
  $\PFA(S)$ implies $\GMP$.
\end{corollary}

\begin{proof}
  Suppose that $\PFA(S)$ holds in $V$.
  Fix a regular $\theta \geq \omega_2$. By the proof of Theorem \ref{pfa_s_s_isp_thm}, 
  there are stationarily many $M \in \power_{\omega_2} H(\theta)$ such that $M^S$ is 
  an $\omega_1$-guessing model in $V^{S}$. By Proposition \ref{guessing_pullback_prop}, 
  for each such $M$ for which $M \prec H(\theta)$ and $S \cup \{S\} \subseteq M$, 
  we know that $M$ is a guessing model in $V$. Therefore, $\GMP$ holds in $V$.
\end{proof}

\subsection{Almost Suslin trees}

As mentioned at the end of Section~\ref{special_sec}, work of Cox and Krueger 
shows that $\IGMP$ follows from the conjunction 
of $\GMP$ and the assertion that all trees of size and height $\omega_1$ with no 
cofinal branches are special, and that $\IGMP$ implies that there are no Suslin trees. 
Together, these results raise a natural question, asked already in 
\cite{cox_krueger_indestructible}, as to whether $\IGMP$ implies that all 
trees of height and size $\omega_1$ with no cofinal branches are special. 
In this section, we prove that $\PFA(T^*)$ implies $\IGMP$. Since every special 
tree contains a stationary antichain, this answers the question negatively. 
It also positively answers a question of Krueger from \cite{krueger_forcing_axiom} 
as to whether $\PFA(T^*)$ implies $\neg \wKH$, since, as mentioned in Section 
\ref{kurepa_sec}, $\GMP$ implies $\neg \wKH$.

The following proposition is a slight improvement of 
\cite[Proposition 4.4 and Corollary 4.5]{cox_krueger_indestructible}. 
To simplify the statement, 
let us say that a tree of height $\omega_1$ is \emph{persistently branchless} 
if it has no cofinal branches and continues to have no cofinal branches in any 
forcing extension that preserves $\omega_1$. Note that every special tree $T$ is 
persistently branchless, and, in fact, it is enough that $T_C$ be special for 
some cofinal set $C \subseteq \omega_1$.

\begin{proposition} \label{persistently_branchless_prop}
  Suppose that $\GMP$ holds and every tree $T$ of size and height $\omega_1$ with 
  no cofinal branches is persistently branchless. Then $\IGMP$ holds.
\end{proposition}

\begin{proof}
  Fix a regular cardinal $\theta \geq \omega_2$. It suffices to show that every 
  $\omega_1$-guessing model $N \in \power_{\omega_2} H(\theta)$ with 
  $\omega_1 \subseteq N$ is an indestructible $\omega_1$-guessing model. To this 
  end, fix such an $N$, and fix a cardinal $\lambda \in N$. It suffices to prove 
  that, in any forcing extension of $V$ in which $\omega_1$ is preserved, 
  $N$ continues to be an $\omega_1$-guessing model for $\lambda$.
  
  Since $N$ is an $\omega_1$-guessing model with $\omega_1 \subseteq N$, 
  \cite[Theorem 1.4]{krueger_sch} implies that $N$ is internally unbounded. 
  We can therefore fix a $\subseteq$-increasing sequence $\langle X_i \mid 
  i < \omega_1 \rangle$ such that each $X_i$ is an element of 
  $N \cap \power_{\omega_1} \lambda$ and $\bigcup_{i < \omega_1} X_i = N \cap \lambda$.
  We now define a tree $T$ as in the proof of 
  \cite[Proposition 4.4]{cox_krueger_indestructible}. The underlying set of $T$ 
  is all pairs $(i,x) \in N$ such that 
  $i < \omega_1$ and $x \subseteq X_i$. We set 
  $(i,x) <_{T} (j,y)$ if and only if $i < j$ and $y \cap X_i = x$.
  
  Precisely as in \cite{cox_krueger_indestructible}, $T$ is a tree of size and 
  height $\omega_1$ with at most $\omega_1$-many cofinal branches. Let $B$ be the 
  set of all cofinal branches through $T$, and, using Lemma 
  \ref{baumgartner_function_lemma}, find a Baumgartner function $g:B \rightarrow T$.
  
  Now, as in the proof of Lemma \ref{technical_lemma}, let $T_0 := \{t \in T \mid 
  \exists b \in B ~ g(b) <_T t \in b\}$, and let $T_1 := T \setminus T_0$. 
  Since $T_0$ contains a tail of every $b \in B$, it follows that $T_1$ has 
  no cofinal branches and is therefore, by assumption, persistently branchless.
  
  Now let $W$ be any forcing extension of $V$ in which $\omega_1$ is preserved, 
  and let $d \in \power^W \lambda$ 
  be $(\omega_1,N)$-approximated in $W$. Then, for each $i < \omega_1$, we have 
  $(i, d \cap X_i) \in T$, and the set $c = \{(i, d \cap X_i) \mid i < 
  \omega_1\}$ is a cofinal branch through $T$.
  Since $T_1$ is persistently branchless in $V$, it must be the case that a
  tail of $c$ lies inside $T_0$. Then, precisely by the argument in the proof 
  of Claim~\ref{unctble_branch_claim}, we must in fact have $c \in B$, i.e., 
  $c$ is already in $V$. Then $d \cap N = \bigcup_{i < \omega_1} d \cap X_i$ is also 
  in $V$ and is clearly $(\omega_1, N)$-approximated there, since $d$ is 
  $(\omega_1, N)$-approximated in $W$. Therefore, 
  since $N$ is a guessing model in $V$, $d \cap N$ is $N$-guessed, i.e., 
  there is $e \in N$ such that $e \cap N = d \cap N$. Clearly, this set $e$ 
  witnesses that $d$ is $N$-guessed in $W$. Therefore, $N$ is an $\omega_1$-guessing 
  model in $W$; since $W$ was arbitrary, $N$ is an indestructible 
  $\omega_1$-guessing model in $V$.
\end{proof}

For the rest of this subsection, let $T^*$ denote an almost Suslin Aronszajn tree.
In \cite[Definition 1.7]{krueger_forcing_axiom}, Krueger introduces the notion of 
a forcing poset $\bb{P}$ being \emph{$T^*$-proper}. Although being $T^*$-proper 
is a weakening of the conjunction of being proper and $T^*$-preserving, 
Krueger proves that $\PFA(T^*)$ is equivalent to $\mathsf{FA}(\mc C)$, where 
$\mc C$ is the class of $T^*$-proper forcing posets 
\cite[Proposition 2.5]{krueger_forcing_axiom}. We will not need the definition 
of $T^*$-properness here, but just note the following facts:
\begin{fact} \label{t_proper_fact}
\begin{enumerate}
  \item \cite[Theorem 2.2]{krueger_forcing_axiom} $T^*$-properness is preserved 
  under countable support iteration.
  \item \cite[Proposition 3.1]{krueger_forcing_axiom} Every strongly proper 
  forcing poset is $T^*$-proper. In particular, Cohen forcing is $T^*$-proper.
  \item \cite[Proposition 1.11]{krueger_forcing_axiom} Every $\omega_1$-closed 
  forcing poset is $T^*$-proper.
  \item \label{shelah_poset} 
  \cite[\S IX, Lemma 4.6]{proper_and_improper_forcing} For every tree 
  $T$ of size and height $\omega_1$ with no cofinal branches, there is a 
  $T^*$-proper forcing $\Q(T)$ that adds an unbounded subset $A \subseteq \omega_1$ 
  and a specializing function $f:T_A \rightarrow \omega$.
\end{enumerate}
\end{fact}
As noted at the end of \cite{krueger_forcing_axiom}, 
Fact \ref{t_proper_fact}(\ref{shelah_poset}) shows 
that $\PFA(T^*)$ implies that for every tree $T$ of size and height $\omega_1$ 
with no cofinal branches, there is an unbounded $A \subseteq \omega_1$ such 
that $T_A$ is special; in particular, $T$ is persistently branchless. Therefore, 
by Proposition~\ref{persistently_branchless_prop}, to show that $\PFA(T^*)$ 
implies $\IGMP$, it suffices to show that $\PFA(T^*)$ implies $\GMP$.

\begin{theorem}
  $\PFA(T^*)$ implies $\IGMP$.
\end{theorem}

\begin{proof}
  Suppose that $\PFA(T^*)$ holds. As noted immediately before the statement of 
  the theorem, it suffices to prove that $\GMP$ holds. Since all of the ideas of 
  this proof are transparently present in the arguments of 
  Subsection~\ref{suslin_subsec} of this paper and Section 4 of 
  \cite{viale_weiss}, we provide only a sketch.
  
  \begin{claim} \label{technical_claim}
    Suppose that $\lambda$ is an uncountable cardinal and $\theta$ is a 
    sufficiently large regular cardinal. Assume that $\P$ satisfies the 
    $\omega_1$-covering and $\omega_1$-approximation properties and collapses 
    $2^\lambda$ to have size $\omega_1$. Then there is a $\P$-name $\dot{T}_1$ for a 
    tree of height and size $\omega_1$ with no cofinal branches and a 
    $w \in H(\theta)$ such that, letting $\dot{\Q}(\dot{T}_1)$ be a name for 
    the poset in Fact \ref{t_proper_fact}(\ref{shelah_poset}), we have the 
    following: for every $M \in \power_{\omega_2} H(\theta)$ such that 
    $\omega_1 \cup \{w\} \subseteq M \prec H(\theta)$, if there is 
    $G \subseteq \P \ast \dot{\Q}(\dot{T}_1)$ that is $M$-generic, then 
    $M$ is a guessing model for $\lambda$.
  \end{claim}
  
  \begin{proof}
    The proof is exactly as in the proof of \cite[Lemma 4.6]{viale_weiss}, 
    all of the ideas of which are also present in the proof of \ref{technical_lemma} 
    above, so we omit it. The only difference between the proof of this claim 
    and that of \cite[Lemma 4.6]{viale_weiss} is that in that paper the forcing 
    $\dot{\Q}(\dot{T}_1)$ is replaced by the ccc forcing to fully specialize 
    $\dot{T}_1$. However, it is evident from that proof that adding this 
    complete specialization function is not necessary; it suffices to specialize 
    $\dot{T}_1$ restricted to some cofinal set of levels, which is precisely 
    what $\dot{\Q}(\dot{T}_1)$ does.
  \end{proof}
  
  As in the proof of \ref{pfa_s_s_isp_thm}, it suffices to show that 
  $\GMP^\lambda$ holds for all uncountable cardinals $\lambda$. Therefore, 
  fix such a $\lambda$. Let $\P := \mathrm{Add}(\omega, 1) \ast 
  \dot{\mathrm{Coll}}(\omega_1, 2^\lambda)$. Again as in the proof of 
  \ref{pfa_s_s_isp_thm}, $\P$ has the $\omega_1$-covering and $\omega_1$-approximation 
  properties and collapses $2^\lambda$ to have size $\omega_1$. We can therefore 
  apply Claim \ref{technical_claim} to find a $\P$-name $\dot{T}$ for a tree of 
  height and size $\omega_1$ with no cofinal branches and a $w \in H(\theta)$ 
  such that, for every $M \in \power_{\omega_2} H(\theta)$ such that 
  $\omega_1 \cup \{w\} \subseteq M \prec H(\theta)$, if there is 
  $G \subseteq \P \ast \dot{\Q}(\dot{T})$ that is $M$-generic, then 
  $M$ is a guessing model for $\lambda$.
  
  Since $\P$ is a two-step iteration in which the first iterand is strongly 
  proper and the second is forced to be $\omega_1$-closed, Fact \ref{t_proper_fact} 
  implies that $\P$ is $T^*$-proper. Also by Fact \ref{t_proper_fact}, 
  $\dot{\Q}(\dot{T})$ is forced by $\P$ to be $T^*$-proper, so $\P \ast 
  \dot{\Q}(\dot{T})$ is $T^*$-proper. Therefore, by $\PFA(T^*)$ and 
  Lemma~\ref{woodin_lemma}, the set of $M \in \power_{\omega_2} H(\theta)$ 
  such that $\omega_1 \cup \{w\} \subseteq M \prec H(\theta)$ and there exists 
  an $M$-generic $G \subseteq \P \ast \dot{\Q}(\dot{T})$ is stationary in 
  $\power_{\omega_2} H(\theta)$. But then, by the previous paragraph, every 
  element of this stationary set is a guessing model for $\lambda$, so 
  $\GMP^\lambda$ holds, and, since $\lambda$ was arbitrary, $\GMP$ holds, 
  as desired.
\end{proof}

Our result shows that $\IGMP$ does not imply that all trees of size and height $\omega_1$ with no cofinal branches are special. However, $\IGMP$ implies that every such tree is special on a dense subset. 

\begin{proposition}
Assume $\IGMP$. Let $T$ be a tree of size and height $\omega_1$. Then there is a dense subset $S$ of $T$ such that $S$ is special.
\end{proposition}

\begin{proof}
Assume $\IGMP$. Let $T$ be a tree of size and height $\omega_1$. Note that the converse of Proposition \ref{persistently_branchless_prop} holds as well: $\IGMP$ implies that every tree of size and height $\omega_1$ with no cofinal branches is persistently branchless by \cite[Theorem 3.6]{cox_krueger_indestructible}. Hence $T$ is persistently branchless and therefore $T$ is not $\omega_1$-distributive. 

Let $\dot{g}$ be a $T$-name for a new function from $\omega$ into the ordinals. Using the name $\dot{g}$ we can define a dense subset of $T$ containing all nodes of $T$ which decide the value of $\dot{g}(n)$ for some $n\in\omega$ and are minimal such; formally $S=\set{t\in T}{\exists n\in\omega ((t\parallel \dot{g}(n)) \mbox{ and } \forall s<t (s \nparallel\dot{g}(n)))}$. Note that $S$ is dense since $\dot{g}$ is a name for a new function. It is straightforward to define a specialization function $f$ from $S$ to $\omega$: for $s\in S$, set $f(s)=n$, where $n$ is the minimal $n\in\omega$ such that $s$ decides the value of $\dot{g}(n)$ and no $t<s$ decides the value of $\dot{g}(n)$. It is ease to verify that $f$ is a specialization function for $S$.
\end{proof}

\bibliographystyle{plain}
\bibliography{bib}

\end{document}